\documentclass[a4paper,12pt]{article}
\usepackage[english]{babel}
\usepackage{amsmath,amsthm,amssymb}
\usepackage{mathrsfs} 
\usepackage{mathabx} 
\usepackage{graphicx}
\usepackage{xcolor}
\usepackage{bm}

\theoremstyle{plain}
\newtheorem{theorem}{Theorem}[section]
\newtheorem*{theorem*}{Theorem}
\newtheorem{lemma}[theorem]{Lemma}
\newtheorem{corollary}[theorem]{Corollary}
\newtheorem{proposition}[theorem]{Proposition}
\newtheorem*{proposition*}{Proposition}

\theoremstyle{definition}
\newtheorem{definition}[theorem]{Definition}
\theoremstyle{definition}
\newtheorem{example}[theorem]{Example}

\theoremstyle{remark}
\newtheorem{remark}[theorem]{Remark}

\numberwithin{equation}{section}

\DeclareMathOperator{\supp}{supp}

\newcommand{\abs}[1]{\left\lvert #1 \right\rvert}
\newcommand{\norm}[1]{\left\lVert #1 \right\rVert}

\newcommand{\R}{\mathbb{R}}
\newcommand{\C}{\mathbb{C}}
\newcommand{\N}{\mathbb{N}}

\newcommand{\f}[2]{\mathscr{F}_{#1}{#2}}
\newcommand{\finv}[2]{\mathscr{F}_{#1}^{-1}{#2}}

\newcommand{\D}{D}

\newcommand{\tp}[1]{\ensuremath{{#1}_{\Theta^{\perp}}}}
\newcommand{\tpk}[1]{\ensuremath{{#1}_{\Theta_{k}^{\perp}}}}
\newcommand{\ftp}{\ensuremath{{\mathscr{F}}_{\Theta^{\perp}}}}
\newcommand{\ftpk}{\ensuremath{{\mathscr{F}}_{\Theta_{k}^{\perp}}}}
\newcommand{\ftpj}[1]{\ensuremath{{\mathscr{F}}_{\Theta_{#1}^{\perp}}}}
\newcommand{\na}[1]{\noalign{\vspace*{0.2cm}\noindent{#1}\vspace*{0.2cm}}}
\renewcommand{\L}[1]{L^{#1}}
\newcommand{\nn}{\nonumber}

\newcommand{\Dj}{\frac{\partial}{\partial x_j}}
\newcommand{\Int}{\ensuremath{\mathop{\int}\limits}}

\makeatletter
\def\blfootnote{\xdef\@thefnmark{}\@footnotetext}
\makeatother

\title{Translation-Invariant Estimates for
  Operators with Simple Characteristics} \author{ Eemeli
  Bl{\aa}sten \thanks{HKUST Jockey Club Institute for Advanced Study,
    The Hong Kong University of Science and Technology, 
    Clear Water Bay, Kowloon, Hong Kong
    (email: {\tt iaseemeli@ust.hk}, telephone: {\tt+852 3469 2320},
    fax: {\tt+852 2243 1538}).}
  \and John Sylvester \thanks{Department of Mathematics, University
    of Washington, Seattle, Washington 98195, U.S.A. ({\tt
    sylvest@uw.edu}).}}

\date{}
\begin{document}
\maketitle

\begin{abstract}
  We prove \(L^{2}\) estimates and solvability for a variety of simply
  characteristic constant coefficient partial differential
  equations  \(P(D)u=f\).  These estimates
\[
||u||_{L^2(D_{r})}\le C\sqrt{d_{r}d_{s}} ||f||_{_{L^2(D_{s})}}
\]
depend on geometric quantities --- the diameters \(d_{r}\)
and \(d_{s}\) of the regions \(D_{r}\), where we estimate
\(u\), and \(D_{s}\), the support of \(f\) --- rather than
weights. As these geometric quantities transform simply
under translations, rotations, and dilations, the
corresponding estimates share the same properties. In
particular, this implies that they transform appropriately
under change of units, and therefore are physically
meaningful.  The explicit dependence on the diameters
implies the correct global growth estimates. The weighted
\(L^{2}\) estimates first proved by Agmon \cite{Agmon} in
order to construct the generalized eigenfunctions for
Laplacian plus potential in \(\R^{n}\), and the more
general and precise Besov type estimates of Agmon and
H\"ormander \cite{Agmon-Hormander}, are all simple direct
corollaries of the estimate above.
\end{abstract}

\section*{Aknowledgements}
  E. Bl{\aa}sten was partially supported by the European Research 
  Council's 2010 Advanced Grant 267700. J. Sylvester was partially 
  supported by the National Science Foundation's grant DMS-1309362.

\newpage
\section{Introduction}

Constant coefficient partial differential equations are
translation invariant, so it is natural to seek estimates
that share this property.  For the Helmholtz equation, and
other equations related to wave phenomena, \(L^2\)-norms
are appropriate in bounded regions because they measure
energy.  For problems in all of \(\R^{n}\),
however, a solution with finite \(L^2\)-norm
may radiate infinite power\footnote{Finite radiated power
  typically means that solutions decay fast enough at
  infinity. For outgoing solutions to the Helmholtz
  equation, radiated power can be expressed as the limit
  as \(R\rightarrow\infty\)
  of the \(L^{2}\)
  norm of the restriction of the solution to the sphere of
  radius \(R\).
  It remains finite as long as solutions decay as
  \(r^{-\frac{n-1}{2}}\)
  in \(n\)
  dimensions.}, and therefore not satisfy the necessary
physical constraints. The solution provided by Agmon
\cite{Agmon} was to introduce \(L^2_{\delta}\)
spaces where weights \((1+|x|^{2})^{\frac{\delta}{2}}\)
correctly enforced the finite transmission of power, but
gave up the translation invariance, as well as scaling
properties necessary for the estimates to make sense in
physical units. Later work by Agmon
and H\"ormander \cite{Agmon-Hormander} used Besov spaces
to exactly characterize solutions that radiated finite
power, but these spaces also relied on a weight and
therefore broke the translation invariance that is
intrinsically associated with both the physics and the
mathematics of the underlying problem. Later work by
Kenig, Ponce, and Vega \cite{MR1230709} modified the
Agmon-H\"ormander norms to regain better scaling
properties.

Our goal here is to offer \(L^2\) estimates that enforce finite 
radiation of power without using weights that destroy translation 
invariance and scaling properties. The 
following theorem, which applies to  a class of scalar pde's with constant 
coefficients and simple characteristics, summarizes our
main results, which will be proved as
Theorem \ref{th:main} and Theorem \ref{th:other}.

\begin{theorem*}
    Let \(P(\D)\) be a constant coefficient partial differential 
    operator on \(\R^n\). Assume that it is either
  \begin{enumerate}
    \item real, of second order, and with no real double 
      characteristics, or
    \item of $N$-th order, $N \geqslant 1$, with admissible symbol (Definition
      \ref{admissibleSymbol}) and no complex double characteristics
  \end{enumerate}
    Then there exists a constant
    \(C(P,n)\) such that, for every open bounded
    \(D_{s}\subset\R^{n}\), and every  \(f\in L^2(D_{s})\),
    there is a  \(u\in L^2_{loc}(\R^{n})\) satisfying
 \begin{eqnarray*}
  P(\D)u=f
\end{eqnarray*}
and for any bounded domain \(D_{r}\subset\R^{n}\)
\begin{eqnarray}\label{eq:53}
  ||u||_{L^2(D_{r})}\le C \sqrt{d_{r}d_{s}}||f||_{L^2(D_{s})}
\end{eqnarray}
where  \(d_{j}\) is the diameter of  \(D_{j}\),
the supremum over all lines of the length of the
intersection of the line with  \(D_{j}\); i.e.   
\begin{eqnarray*}
  \nn
  d_{j}= \sup_{\mathrm{lines}\ l}\mu_{1}(l\cap D_{j}).
\end{eqnarray*}
\end{theorem*}

If  \(f\) is not compactly supported, but  \(\supp
f\subset \mathop{\cup}\limits_{j=1}^{\infty} B_{j}\) where
each  \(B_{j}\) has finite diameter  \(b_{j}\), then
(\ref{eq:53}) becomes
\begin{equation}\nn
  ||u||_{L^2(D)}\le C\sqrt{d} \sum_{j=1}^{\infty}\sqrt{b_{j}}
  ||f||_{_{L^2(B_{j})}}
\end{equation}
  which we may rewrite as
  \begin{eqnarray}\label{eq:91}
  \sup_{D}\frac{1}{\sqrt{d}}||u||_{L^2(D)}\le
  \sum_{j=1}^{\infty}\sqrt{b_{j}} ||f||_{_{L^2(B_{j})}}.
  \end{eqnarray}

In the special case that \(D\) is a ball with a fixed
center and arbitrary radius; and the \(B_{j}\) include the
ball of radius one and the dyadic spherical shells
\(2^{j}<|x|<2^{j+1}\) for \(j\ge 0\), these are the
estimates of Agmon-H\"ormander in \cite{Agmon-Hormander}.
The weighted
\(L^2_{\delta}\) estimates introduced by Agmon are also
direct consequences of (\ref{eq:53}), so that these
solutions do radiate finite power and are therefore
physically meaningful. The solutions we construct are not
necessarily unique, but include the \textit{physically
  correct} solutions in all the cases we are aware of. For
the Helmholtz equation, for example, the solution which
satisfies the Sommerfeld radiation condition is among
those which satisfy the estimate \eqref{eq:53}.\\

Our estimates do not include the uniform \(L^p\) estimates
for the Helmholtz equation, shown below, which were
derived in \cite{MR0358216} and \cite{MR894584}, and
presented in \cite{ruiz-notes} and \cite{serov-notes}.

\begin{theorem*}[Uniform $L^p$ estimates]
   Let \(k>0\) and
\(\frac{2}{n}\ge
  \frac{1}{p} - \frac{1}{q}\ge\frac{2}{n+1}\)
for \(n\ge3\) and
\(1> \frac{1}{p} -\frac{1}{q}\ge\frac{2}{3}\)
for \(n=2\), where
  \(\frac{1}{q} + \frac{1}{p}=1\).
There exists a constant \(C(n,p)\), independent of \(k\), such
that, for smooth compactly supported \(u\)
\begin{equation}
\label{eq:84}
  ||u||_{L^{q}(\R^{n})}\le C(n,p)k^{n(\frac{1}{p}
    -\frac{1}{q})-2}||(\Delta+k^{2})u||_{L^{p}(\R^{n})}
\end{equation}
\end{theorem*}

The estimates for the Helmholtz equation in \eqref{eq:84}
share all the invariance properties of \eqref{eq:53}, and
are stronger for small scatterers and applications to
nonlinear problems. The dependence on the wavenumber
\(k\), however, is not as well-suited to applications where the
sources are supported on sets that are several 
wavelengths in size and located far apart, nor do they
have a direct physical interpretation in terms of power.
Additionally, it seems reasonable that the estimate of the
solution in the higher  \(L^{p}\) norm indicates a gain in
regularity. Our methods don't require, or make use of ellipticity,
so we don't expect to recover these estimates.\\

Our methods make use of certain anisotropic norms
introduced in \cite{scaleinvariant} for the Helmholtz
equation. Those estimates were scale and translation
invariant, but, due to the anisotropy, not rotationally
invariant. We show here that a consequence of these mixed
norm estimates is \eqref{eq:53}, which is rotationally
invariant and much simpler than the mixed norm estimates
used to derive it.  Because of the generality, the
mixed norms we use here must be slightly different than
those in in \cite{scaleinvariant}, and the techniques
required to treat more general operators are substantially
more complicated.

We treat only operators with simple characteristics
because a \textit{bona fide} real multiple characteristic
(a real \(\eta\in\R^{n}\)
where the symbol \(p(\eta)\)
and \(\nabla p(\eta)\)
vanish simultaneously) will imply that our techniques
cannot succeed. In Section \ref{sec:examples}, we show
that estimates of the form (\ref{eq:53}) cannot hold for
the Laplacian, which has a double characteristic at the
origin. 

For a single second order operator with real
constant coefficients we will show in Theorem
\ref{th:main} that the absence of multiple characteristics
is sufficient to conclude the estimate
(\ref{eq:53}). Under some additional hypotheses, we will
prove the same estimate for some higher order operators in
Theorem \ref{th:other}.  Additionally, we will prove the
estimate (\ref{eq:53}) for the 4x4 Dirac system, and for a
scalar 4th order equation where H\"ormander's
\textit{uniformly simply characteristic} condition fails.

\section{The Helmholtz case}

We will illustrate our methods by outlining
the proof of (\ref{eq:53}) for the outgoing solution to
the Helmholtz equation below.
\begin{eqnarray}
  \label{eq:13}
  (\Delta + k^2)u = f
\end{eqnarray}
We will choose a direction \(\Theta\) and write
\(x=t\Theta+\tp{x}\). We next Fourier transform in the
\(\Theta^\perp\) hyperplane to rewrite (\ref{eq:13}) as
an ordinary differential equation. We use the notation
\(\f{\Theta^\perp}{u}(t\Theta+\tp{\xi})\) to indicate this
partial Fourier transform (see \eqref{eq:68} below for a
formal definition). If we  set  \(g(t,\tp{\xi}) =
\f{\Theta^\perp}{f}(t\Theta+\tp{\xi})\)  and  \(w(t,\tp{\xi}) =
\f{\Theta^\perp}{u}(t\Theta+\tp{\xi})\), then \eqref{eq:13} becomes
\begin{eqnarray}
  \label{eq:59}
  (\partial_t^2 + k^2 - \abs{\xi_{\Theta^\perp}}^2)
  w = g
\end{eqnarray}

We factor the second order operator as a product
of first order operators
\begin{eqnarray}\nn
\left(\partial_t + i\sqrt{k^2 - \abs{\xi_{\Theta^\perp}}^2}\right) 
  \left(\partial_t - i\sqrt{k^2 - \abs{\xi_{\Theta^\perp}}^2}\right) 
  w = g
\end{eqnarray}
and define a solution  \(w=w_{1}+w_{2}\) where  \(w_{1}\) and  \(w_{2}\) solve  
\begin{eqnarray}\label{modelProblems}
\begin{cases}
  \Big(\partial_t + i\sqrt{k^2 - \abs{\xi_{\Theta^\perp}}^2}\Big) 
  w_{1} = \frac{i g}{ 2\sqrt{k^2 - 
  \abs{\xi_{\Theta^\perp}}^2}}, 
\\ 
\Big(\partial_t - i\sqrt{k^2 - 
  \abs{\xi_{\Theta^\perp}}^2}\Big) w_{2} = \frac{-i 
  g}{ 2\sqrt{k^2 - \abs{\xi_{\Theta^\perp}}^2}}.
\end{cases}
\end{eqnarray}
The solutions  \(w_{1}\) and  \(w_{2}\) are  given by the exact formulas 
\begin{eqnarray}\label{eq:60}
 w_{1}(t,\tp{\xi})&=&\frac{1}{2\sqrt{k^2 -
        \abs{\xi_{\Theta^\perp}}^2}}
    \Int_{-\infty}^{t}
    e^{i\sqrt{k^2 -\abs{\xi_{\Theta^\perp}}^2}(t-s)}\
    i g(s,\tp{\xi})ds
\\\label{eq:69}
w_{2}(t,\tp{\xi})&=&\frac{1}{2\sqrt{k^2 -
        \abs{\xi_{\Theta^\perp}}^2}}
    \Int_{t}^{\infty}
    e^{-i\sqrt{k^2 -\abs{\xi_{\Theta^\perp}}^2}(t-s)}\
    i g(s,\tp{\xi})
        ds
\end{eqnarray}
The square root
\(\sqrt{k^2-\abs{\xi_{\Theta^\perp}}^2}\)
is chosen so that it always has positive imaginary part
for imaginary part of \(k\) positive, and extends continuously,
as a function of  \(k\),  to the real axis. This insures
that the exponential in (\ref{eq:60}) and (\ref{eq:69}) is
bounded by one\footnote{This also selects the unique outgoing solution, which
  satisfies the Sommerfeld radiation condition.}
so that  \(\f{\Theta^\perp}{u} = w = w_{1} + w_{2}\)  satisfies
\begin{eqnarray}
  \label{eq:61}
  |\f{\Theta^\perp}{u}(t\Theta+\tp{\xi})|\le
  \frac{||\f{\Theta^\perp}{f}(s\Theta+\tp{\xi})||_{\L1(ds)}}
  {\sqrt{k^2-\abs{\xi_{\Theta^\perp}}^2}} 
\end{eqnarray}
which would yield a simple estimate if the denominator had
a lower bound.

In sections \ref{sec:order2} and  \ref{sec:ordern}, we
will construct Fourier multipliers that implement a
partition of unity that 
decomposes  \(f\) into a sum 
\begin{eqnarray}
  f&=&f_{1}+f_{2}+\ldots+f_{m}
\\  \label{eq:85}
 &=&f\phi_{1}+f\phi_{2}(1-\phi_{1})
     +f\phi_{3}(1-\phi_{2})(1-\phi_{1})\ldots+f\prod_{j=1}^{m}(1-\phi_{j})
\end{eqnarray}
such that, for each \(f_{j}\), there is a direction
\(\Theta_{j}\) such that  
\begin{eqnarray}
  \label{eq:65}
  \inf_{\supp
  \f{\Theta_{j}^\perp}{f}}\sqrt{k^2-\abs{\xi_{\Theta_{j}^\perp}}^2}
  &\ge& \varepsilon k
\\\na{and}\nn
\norm{\norm{\f{\Theta_{j}^\perp}{f_{j}}(s\Theta_j+\xi_{\Theta_j^\perp})}_{\L1(ds)}}_{\L2(d\xi_{\Theta_j^\perp})}
&\le&\nn
  \norm{\norm{\f{\Theta_{j}^\perp}{f}(s\Theta_j+\xi_{\Theta_j^\perp})}_{\L1(ds)}}_{\L2(d\xi_{\Theta_j^\perp})}
\\\na{which we write more compactly as}\label{eq:67}
\norm{\f{\Theta_{j}^\perp}{f_{j}}}_{\Theta_{j}(1,2)}&\le&\norm{\f{\Theta_{j}^\perp}{f}}_{\Theta_{j}(1,2)}
\end{eqnarray}
using norms which we will define precisely in
(\ref{eq:76}).

\begin{figure}
\begin{center}
  \includegraphics{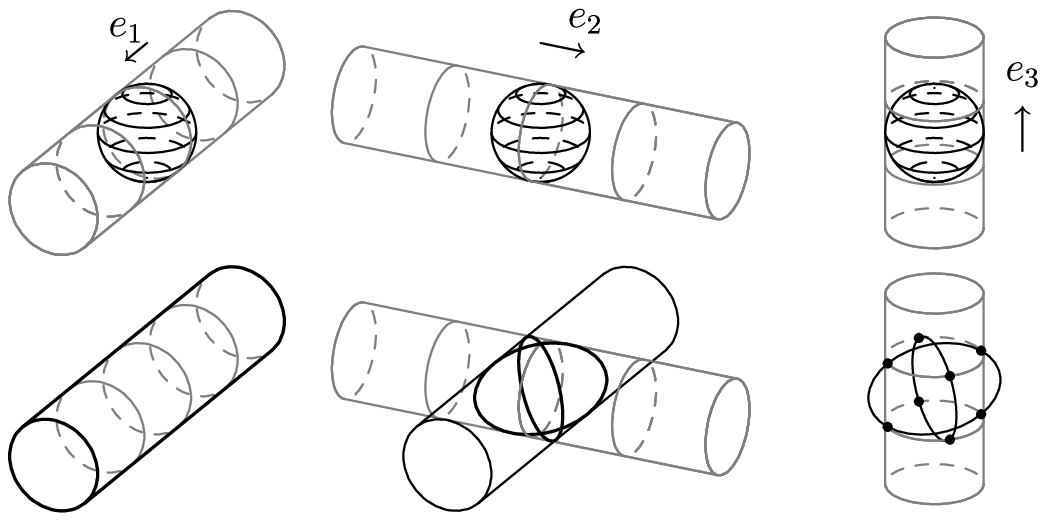}
\end{center}
\caption{Partition of unity for $\Delta+k^2$.}
\label{partitionUnityProcedureFig}
\end{figure}

We illustrate this decomposition for the 3-dimensional
case in Figure \ref{partitionUnityProcedureFig}. Let
$\Theta_j=e_j$, $j=1,2,3$, be an orthogonal basis. The
cylinders illustrated in the
top-row are the sets,  denoted $\mathscr{B}_{\Theta_{j},0}$, where the
denominators $\sqrt{k^2 - \abs{\xi_{\Theta_{j}^\perp}}^2}$
of \eqref{eq:61} vanish.  Each \(\phi_{j}\), and hence each \(f_{j}\),
 vanishes in a neighborhood of
$\mathscr{B}_{\Theta_{j},0}$. The thick lines in the
figures in the bottom row show the intersections
$\mathscr{B}_{\Theta_1,0}$,
$\mathscr{B}_{\Theta_1,0} \cap \mathscr{B}_{\Theta_2,0}$
and
$\mathscr{B}_{\Theta_1,0} \cap \mathscr{B}_{\Theta_2,0}
\cap \mathscr{B}_{\Theta_3,0}$, indicating the
support of the \(\prod_{j=1}^{m}(1-\phi_{j})\). To
guarantee that the  \(f_{j}\) sum to  \(f\), the
intersection of (neighborhoods of) all the 
\(\mathscr{B}_{\Theta_{j},0}\) must be empty. We see in the figure that
the intersection of the first three neighborhoods consists
of neighborhoods of eight points, so we may add a fourth direction, 
for example
$\Theta_4 = (e_1+e_2+e_3)/\sqrt{3}$ (not pictured), so that the
corresponding cylinder $\mathscr{B}_{\Theta_4,0}$ does not
intersect the eight points that are left.\\

Combining \eqref{eq:61}, \eqref{eq:65}, and \eqref{eq:67}
\begin{eqnarray}
  \label{eq:66}
  ||\f{\Theta_{j}^\perp}{u_{j}}||_{\Theta_{j}(\infty,2)}\le
  \frac{||\f{\Theta_{j}^\perp}{f}||_{\Theta_{j}(1,2)}}{k\varepsilon} 
\end{eqnarray}
where each of the  \(u_{j}\) solves  \((\Delta+k^{2})u_{j}=f_{j}\).  
The estimates (\ref{eq:66}) estimate each \(u_{j}\) in a
different norm, and the norms, which  depend on a choice of
 the vectors \(\Theta_{j}\), are no longer
rotationally invariant. They can, however, be combined to
yield an estimate in a single norm that is
rotationally and translationally invariant.
\begin{lemma}\label{mixedToUniformNorms}
  Let $D_s, D_r \subset \R^n$ be domains with diameters $d_s$ and $d_r$, 
  respectively. Let $\f{\Theta^\perp}{u} \in \Theta(\infty,2)$ and $f 
  \in L^2_{loc}$. Assume that $\supp f \subset D_s$. Then $u_{\mid D_r} 
  \in L^2$ and $\f{\Theta^\perp}{f} \in \Theta(1,2)$. Moreover
  \begin{eqnarray}
\nn
  \norm{u}_{L^2(D_r)} \le \sqrt{d_r} \norm{ \f{\Theta^\perp}{u} }_{\Theta(\infty,2)},    
\\
    \label{eq:71}
  \norm{ \f{\Theta^\perp}{f} }_{\Theta(1,2)} \le \sqrt{d_s} \norm{ f }_{L^2(D_s)}.  
  \end{eqnarray}
\end{lemma}
\noindent Combining the lemma with (\ref{eq:66}) yields
\begin{eqnarray}
  \label{eq:6}
  \norm{u}_{L^2(D_r)} \le
  \frac{\sqrt{d_rd_{s}}}{\varepsilon k}\norm{ f }_{L^2(D_s)}.
\end{eqnarray}
We leave the proof of the lemma for the next section,
after we have given the formal definitions of the norms.

\section{Mixed norms}

  We begin with the formal definition of the anisotropic
  norms we will use.

\begin{definition}\label{variableSplittingDef}
  Let $\Theta\in\mathbb{S}^{n-1}(\R^n)$. We split any
  $x\in\R^n$ as 
  \begin{eqnarray}
    \label{eq:83}
    x = t\Theta+x_{\Theta^\perp}
  \end{eqnarray}
where
  $t = x\cdot\Theta$ and
  $x_{\Theta^\perp} = x - (x\cdot\Theta)\Theta$. We split
  the dual variable $\xi$ as
  \begin{eqnarray}
\nn
    \xi = \tau \Theta + \xi_{\Theta^\perp}
  \end{eqnarray}
   The variables
  $t$ and $\tau$ are dual, and so are $x_{\Theta^\perp}$
  and $\xi_{\Theta^\perp}$.
\end{definition}  

\begin{figure}
\begin{center}
  \includegraphics{Fig2}
\end{center}
\caption{Splitting of $\xi = \tau\Theta+\tp{\xi}$.}
\end{figure}

\begin{definition}\label{fourierTransformDef}
  By $\f{\Theta}{}$ we denote the one-dimensional Fourier transform along 
  the direction $\Theta$. If $f\in\mathscr{S}(\R^n)$ then
\[
  \f{\Theta}{f}(\tau\Theta+x_{\Theta^\perp}) = \frac{1}{\sqrt{2\pi}} 
  \int_{-\infty}^\infty e^{-it\tau} f(t\Theta+x_{\Theta^\perp}) dt
\]
  using the notation of Definition \ref{variableSplittingDef}. The 
  Fourier transform in the orthogonal space $\Theta^\perp$ is denoted by 
  $\f{\Theta^\perp}{}$ and it acts by
  \begin{eqnarray}\label{eq:68}
  \f{\Theta^\perp}{f}(t\Theta+\xi_{\Theta^\perp}) = 
  \frac{1}{(2\pi)^{(n-1)/2}} \int_{\Theta^\perp} e^{-i x_{\Theta^\perp} 
  \cdot \xi_{\Theta^\perp}} f(t\Theta + x_{\Theta^\perp}) 
  dx_{\Theta^\perp}.
  \end{eqnarray}
  The corresponding inverse transforms are denoted by $\finv{\Theta}{}$ 
  and $\finv{\Theta^\perp}{}$.
\end{definition}

\begin{definition}\label{mixedNormSpaceDef}
  We use $\Theta(p,q)$ we denote the space of $L^p(dt)$-valued 
  $L^q(dx_{\Theta^\perp})$-functions (if the variable is 
  $x=t\Theta+x_{\Theta^\perp}$). More precisely $f\in\Theta(p,q)$ if

  \begin{eqnarray}\label{eq:76}
  \norm{f}_{\Theta(p,q)} = \left( \int_{\Theta^\perp} \left( 
  \int_{-\infty}^\infty \abs{f(t\Theta+x_{\Theta^\perp})}^p dt 
  \right)^{q/p} dx_{\Theta^\perp} \right)^{1/q} < \infty.
  \end{eqnarray}
  with obvious modifications for $p=\infty$ or
    $q=\infty$.

\end{definition}

\begin{remark}
  We make note of the fact that the order is important.
For example,
  we will use the norm
\[
  \norm{f}_{\Theta(1,\infty)} = \sup_{x_{\Theta^\perp}\in\Theta^\perp} 
  \int_{-\infty}^\infty \abs{f(t\Theta+x_{\Theta^\perp})} dt
\]
  in several lemmas. This is clearly not the same as
\[
  \int_{-\infty}^\infty \ \sup_{x_{\Theta^\perp}\in\Theta^\perp} 
  \abs{f(t\Theta+x_{\Theta^\perp})} dt
\]
\end{remark}

We convert estimates in these anisotropic norms
to isotropic  \(L^{2}\) estimates with Lemma
\ref{mixedToUniformNorms}. We give the proof now.

\begin{proof}[Proof of Lemma \ref{mixedToUniformNorms}]
  Let $D^\Theta_s$, $D^\Theta_r$ be the projections of $D_s$ and $D_r$ 
  onto the line $t \mapsto t\Theta$.   We have
\begin{align*}
  &\norm{u}_{L^2(D_r)}^2 \le \int_{D^\Theta_r} \int_{\Theta^\perp} 
  \abs{u(t\Theta+x_{\Theta^\perp})}^2 dx_{\Theta^\perp} dt
  \\
  &\qquad = 
  \int_{D^\Theta_r} \int_{\Theta^\perp} 
  \abs{\f{\Theta^\perp}{u}(t\Theta+\xi_{\Theta^\perp})}^2 
  d\xi_{\Theta^\perp} dt 
  \\
  &\qquad \le \int_{D^\Theta_r} dt
  \int_{\Theta^\perp} \sup_{t'} 
  \abs{\f{\Theta^\perp}{u}(t'\Theta+\xi_{\Theta^\perp})}^2 
  d\xi_{\Theta^\perp} \le d_r \norm{ \f{\Theta^\perp}{u} 
  }_{\Theta(\infty,2)}^2
\end{align*}
where we have used the Plancherel formula and the
hypothesis that the diameter of $D_r$ is at most
\(d_{r}\), which implies that \(D^\Theta_r\) is contained
in a union of intervals of length at most \(d_{r}\). The
proof of (\ref{eq:71}) is similar, and makes use of the
fact that \(D^\Theta_s\) is contained in a union of
intervals of length less than \(d_{s}\).
\[
  \int_{-\infty}^\infty \abs{ 
  \f{\Theta^\perp}{f}(t\Theta+\xi_{\Theta^\perp})} dt \le 
  \sqrt{d_{s}} \norm{ 
  \f{\Theta^\perp}{f}(t\Theta+\xi_{\Theta^\perp}) }_{L^2(dt)}
\]
   The inequality (\ref{eq:71})  follows by taking the 
  $L^2(d\xi_{\Theta^\perp})$-norm and using Fubini's theorem, and then 
  the Plancherel formula.
\end{proof}

\begin{remark}\label{mixedToUniformNormsLp} 
  Let $\frac{1}{q}+\frac{1}{p}=1$ and $q \geqslant 2$. An
  analogous argument shows that
  $\norm{u}_{L^q(D_r)} \le d_r^{1/q} \norm{ \f{\Theta^\perp}{u} 
  }_{\Theta(\infty,p)}$ and $\norm{ \f{\Theta^\perp}{f} }_{\Theta(1,p)} 
  \le d_s^{1/p} \norm{ \f{}{f} }_{L^p(\R^n)}$. 
\end{remark}

\section{Fourier Multiplier Estimates}

\begin{definition}\label{multiplierDef}
  Let $\Psi:\R^n\to\C$ be locally integrable. We define the
  Fourier multiplier $M_\Psi$ as the
  the operator
\[
  M_\Psi f = \finv{}{\{\Psi \f{}{f}\}}.
\]  
\end{definition}

Because our estimates rely on decompositions of sources
similar to \eqref{eq:85} where \(f_{j}= M_{\Psi_{j}}f\)
must satisfy satisfy conditions similar to \eqref{eq:65}
and the estimate \eqref{eq:67}, we need to establish the
boundedness of these Fourier multipliers on the mixed
norms of the partial Fourier transforms of the sources,
i.e.  on \(\norm{\ftp{f(t,\xi)}}_{\Theta(1,2)}\). Our first
lemma tells us that 
\(\norm{\finv{\Theta}{\Psi}}_{\Theta(1,\infty)}<\infty\)
  is enough to guarantee such a bound.

\begin{lemma}\label{multiplierNorm}
  Let $\finv{\Theta}{\Psi} \in \Theta(1,\infty)$. Then
\[
  \norm{ \f{\Theta^\perp}{M_\Psi f} }_{\Theta(1,p)} \le 
  \frac{1}{\sqrt{2\pi}} \norm{ \finv{\Theta}{\Psi} }_{\Theta(1,\infty)} 
  \norm{ \f{\Theta^\perp}{f} }_{\Theta(1,p)}.
\]
\end{lemma}

\begin{proof}
  Write $\f{\Theta^\perp}{M_\Psi f} = \finv{\Theta}{\{ \Psi \f{}{f} \}} 
  = \frac{1}{\sqrt{2\pi}} \finv{\Theta}{\Psi} \ast_t 
  \f{\Theta^\perp}{f}$. Then take the $L^1(dt)$-norm and use Young's 
  inequality for convolutions. The result follows then by taking the 
  $L^p(d\xi_{\Theta^\perp})$-norm.
\end{proof}

Our Fourier multipliers will not be Schwartz class
functions. They will be smooth, but will always be
constant in a direction $\nu$, so the integrability
properties necessary to verify that the
\(\Theta(1,\infty)\) norm is finite may be a bit subtle,
and will depend on the relation between the direction
\(\nu\) of that coordinate and the direction \(\Theta\)
which defines the relevant norm. The estimates will be
simplest when the directions \(\nu\) and \(\Theta\)
coincide, or are perpendicular. Because second order
operators have a convenient normal form, the decompositions in
Section \ref{sec:order2} will only require multipliers
with \(\nu\) and \(\Theta\) either identical or
perpendicular. Higher order operators do not admit such
simple normal forms, so the decompositions are based on
abstract algebraic properties, and we cannot, in general, restrict to
these simple cases. The next proposition, and
its corollary, tell us how to reduce the
\(\Theta(1,\infty)\) estimate for the norm of a multiplier
that is 
constant in the \(\nu\) direction, to the case where
\(\nu\) and  \(\Theta\) are either parallel or perpendicular.\\  
  
We need a little notation first.
  Define $\nu_\perp$ to be a unit vector in the \((\Theta,\nu)\) plane 
  perpendicular to \(\nu\) so that the pair
  \((\nu,\nu_{\perp})\) is positively oriented , and
  define $\Theta_\perp$ analogously to be the
  unit vector in that plane perpendicular to 
  \(\Theta\). Finally, let $\xi_{\perp\perp}$ denote the
  component of  any \(\xi\in\R^{n}\) perpendicular to the
  \((\Theta,\nu)\) plane .

\begin{proposition}\label{fourierTwoDirections}
  Let $\nu \in \mathbb{S}^{n-1}(\R^n)$ and $\psi \in 
  \mathscr{S}(\nu^\perp)$. Define
  \begin{eqnarray}\nn
     \Psi(\sigma\nu + \xi_{\nu^\perp}) = \psi(\xi_{\nu^\perp}) \qquad 
  \forall \sigma \in \R.
  \end{eqnarray}
  If $\Theta \not\parallel \nu$ and $\Theta \cdot \nu = \cos \alpha$, 
  $\alpha \in (0,\pi)$ then, 
\begin{equation}
  \finv{\Theta}{\Psi} (t \Theta + \xi_{\Theta^\perp}) = \frac{e^{i \ell 
  t \cot\alpha}} {\sin\alpha} \finv{\nu_\perp}{\psi} \big( 
  \frac{t}{\sin\alpha} \nu_\perp + \xi_{\perp\perp} \big),
\end{equation}
  where $\ell = \xi_{\Theta^\perp}\cdot\Theta_\perp$ and
  \(\xi_{\perp\perp}\) is the component of  \(\xi=t
  \Theta + \xi_{\Theta^\perp}\) perpendicular to the  \((\Theta,\nu)\) plane.
If $\Theta \parallel \nu$ then
\begin{eqnarray}
\label{eq:87}
  \finv{\Theta}{\Psi} (t\Theta+\xi_{\Theta^\perp}) = \sqrt{2\pi} 
  \delta_0(t) \psi(\xi_{\Theta^\perp}).
\end{eqnarray}
\end{proposition}

\begin{proof}
  According to Definition \ref{fourierTransformDef}
\[
  \finv{\Theta}{\Psi} (t\Theta+\xi_{\Theta^\perp}) = 
  \frac{1}{\sqrt{2\pi}} \int_{-\infty}^\infty e^{it\tau} 
  \Psi(\tau\Theta+\xi_{\Theta^\perp}) d\tau.
\]
It is easy to check that 
\begin{eqnarray*}
  \Theta &=& \cos\alpha \,\nu + \sin\alpha
    \,\nu_\perp, \label{nuTdef}
\\ 
  \nu &=& \cos\alpha \,\Theta + \sin\alpha \,\Theta_\perp, 
  \label{thetaTdef}
\\
  \Theta_\perp &=& \sin \alpha \,\nu - \cos \alpha
                   \,\nu_\perp,
\\\na{and therefore that}
\tau\Theta + \xi_{\Theta^\perp} &=& 
  (\tau\cos\alpha + \ell\sin\alpha)\nu + (\tau\sin\alpha - 
  \ell\cos\alpha)\nu_\perp + \xi_{\perp\perp}.
\end{eqnarray*}
  Because $\Psi$ is 
  constant in the direction $\nu$ and  equal to
  $\psi$ on $\nu^\perp$,
\begin{align*}
  \finv{\Theta}{\Psi} (t\Theta+\xi_{\Theta^\perp}) &= 
  \frac{1}{\sqrt{2\pi}} \int_{-\infty}^\infty e^{it\tau} \psi\big( 
  (\tau\sin\alpha - \ell\cos\alpha)\nu_\perp + \xi_{\perp\perp} \big) 
  d\tau \\ & = \frac{1}{\sqrt{2\pi}} \int_{-\infty}^\infty 
  e^{it\left(\frac{\tau'}{\sin\alpha} + \ell\cot\alpha\right)} 
  \psi(\tau'\nu_\perp + \xi_{\perp\perp}) \frac{d\tau'}{\sin\alpha} \\ 
  & = \frac{e^{i\ell t \cot\alpha}} {\sin\alpha} 
  \finv{\nu_\perp}{\psi} \big( \frac{t}{\sin\alpha}\nu_\perp + 
  \xi_{\perp\perp} \big).
\end{align*}

  If $\Theta\parallel\nu$, then  \(\Psi(t\Theta+\tp{\xi})\)
  is independent of  \(t\), so \eqref{eq:87} follows from
  the fact that the one dimensional Fourier transform of
  the constant function is the Dirac delta.
\end{proof}

\begin{corollary}\label{constantInDirectionNorm}
 With the notation of  Proposition \ref{fourierTwoDirections} and 
  $\Theta\not\parallel\nu$ we have
\begin{eqnarray}\label{eq:88}
  \norm{\finv{\Theta}{\Psi}}_{\Theta(1,\infty)} =
  \norm{\finv{\nu_\perp}{\psi}}_{\nu_\perp(1,\infty)},
\end{eqnarray}
  and therefore
\begin{equation}
  \norm{ \f{\Theta^\perp}{M_\Psi f} }_{\Theta(1,p)} \le 
  \norm{\finv{\nu_\perp}{\psi}}_{\nu_\perp(1,\infty)} \norm{ 
  \f{\Theta^\perp}{f} }_{\Theta(1,p)}.
\end{equation}
  If $\Theta\parallel\nu$, then
\begin{equation}
  \norm{ \f{\Theta^\perp}{M_\Psi f} }_{\Theta(1,p)} \le 
  \sup_{\Theta^\perp} \abs{\psi} \norm{ \f{\Theta^\perp}{f} 
  }_{\Theta(1,p)}.
\end{equation}
\end{corollary}

\begin{remark}
 The  \(\nu_\perp(1,\infty)\) norm which appears in
(\ref{eq:88}) is analogous to the
\(\Theta_\perp(1,\infty)\) norm , but is defined on functions of
one fewer variable. Recall that  \(\psi\) is defined on
the  \(\nu^{\perp}\) hyperplane, and  \(\nu_{\perp}\) is
a unit vector in that hyperplane perpendicular to
\(\Theta\).
Thus  \(\finv{\nu_\perp}{\psi}\) is a function of
\(\tau\nu_{\perp}+ \xi_{\perp\perp}\), and the
\(\nu_\perp(1,\infty)\) norm means
the supremum over  \(\xi_{\perp\perp}\) of the
\(L^{1}(d\tau)\) norm.     
\end{remark}

\begin{proof}
  We have $0<\alpha<\pi$ and so $\sin\alpha>0$. Hence
\[
  \int_{-\infty}^\infty 
  \abs{\finv{\Theta}{\Psi}(t\Theta+\xi_{\Theta^\perp})} dt = 
  \int_{-\infty}^\infty \abs{\finv{\nu_\perp}{\psi}(t'\nu_\perp + 
  \xi_{\perp\perp})} dt'
\]
  by a change of variables. Then we can take the supremum over 
  $\xi_{\Theta^\perp} \in \Theta^\perp$, which will give the same result 
  as the supremum of $\xi_{\perp\perp}$ over $\nu^\perp \cap 
  \nu_\perp^\perp$. The multiplier estimate follows from Lemma 
  \ref{multiplierNorm}. For the second case note that $M_\Psi f = 
  \finv{\Theta^\perp}{\{\psi(\xi_{\Theta^\perp}) 
  \f{\Theta^\perp}{f}\}}$. The claim follows directly.
\end{proof}

\section{Estimates  for 2\textsuperscript{nd} order operators}\label{sec:order2}

  We treat a second order constant coefficient partial
  differential operator \(P(\D)\),
  with no double characteristics, i.e. no
  simultaneous real root of \(p(\xi) = 0\)
  and \(\nabla p(\xi)=0\). The main result of this
  section is:
  \begin{theorem}\label{th:main}
    Let  \(P(\D)\) be a single real second order constant
    coefficient partial differential operator on \(\R^{n}\)
    with no real double
    characteristics. Then there exists a constant
    \(C(P,n)\) such that, for every open bounded
    \(D_{s}\subset\R^{n}\), and every  \(f\in L^2(D_{s})\),
    there is a  \(u\in L^2_{loc}(\R^{n})\) satisfying
 \begin{eqnarray}
  \label{eq:1}
  P(\D)u=f
\end{eqnarray}
such that, for any bounded domain \(D_{r}\subset\R^{n}\)
\begin{eqnarray}
  \label{eq:2}
  ||u||_{L^2(D_{r})}\le C \sqrt{d_{r}d_{s}}||f||_{L^2(D_{s})}
\end{eqnarray}
where  \(d_{j}\) is the diameter of  \(D_{j}\),
the supremum over all lines of the length of the
intersection of the line with  \(D_{j}\); i.e.   
\begin{eqnarray}
  \nn
  d_{j}= \sup_{\mathrm{lines}\ l}\mu_{1}(l\cap D_{j}).
\end{eqnarray}
  \end{theorem}
We begin the proof by writing the second order operator in a simple
normal form.
\begin{lemma}\label{th:1}
  After an orthogonal change of coordinates and a
  rescaling:
  \begin{eqnarray}
    \label{eq:4}
    P(\D) = \sum\epsilon_{j}\left(\Dj\right)^{2}
    + 2\sum\alpha_{j}\Dj + B
  \end{eqnarray}
where  each \(\epsilon_{j}\) equals one of \(0,1,-1\);
\(\alpha_{j}\in\R\), and  \(B\in\R\).    
\end{lemma}
\begin{proof}
  This is a statement about the principal (second order) part $P_2$ of
  the operator. \(P_{2}(\xi)\)
  is a real quadratic form with eigenvalues
  \(-\lambda_{i}\)
  and eigenvectors \(e_{j}\). If we introduce coordinates
\begin{equation}\nn
    x = \sum x_{j}e_{j}
\end{equation}
  then
\[
  P_{2}(\D) = \sum\lambda_{j}\left(\Dj\right)^{2}
\]
After the rescaling
\[
  x_{j}=\sqrt{\lambda_{j}}y_{j}\qquad
  \frac{\partial}{\partial y_{j}}=\sqrt{\lambda_{j}}\Dj
\]
the second order part takes the desired form in (\ref{eq:4}).
\end{proof}

Next, we dismiss the simple cases.
\begin{proposition}\label{prop4.3}
  If some \(\epsilon_{j}=0\) and the corresponding \(\alpha_{j}\ne0\), 
  then Theorem \ref{th:main} is true.
\end{proposition}
\begin{proof}
  Without loss of generality, we may assume that
  \(j=1\). We do a partial Fourier transform in the
  \(x_{1}^{\perp}\) plane, i.e. with
  \(\tilde{x}=(x_{2},\ldots,x_{n})\) and
  \(\xi=(\xi_{2},\ldots,\xi_{n})\). We let \(\Theta_{1}\)
  denote the unit vector in the  \(x_{1}\) direction, and
  let  \(w=\ftp{u}\) and  \(g=\frac{1}{2}\ftp{f}\), i.e.
  \begin{eqnarray}
    \nn
    w(x_{1},\xi) = \frac{1}{(2\pi)^{(n-1)/2}}
    \int_{\R^{n-1}}e^{-i\xi\cdot\tilde{x}}u(x_{1},\tilde{x})dx_{2}\ldots dx_{n}
  \end{eqnarray}
Then  \(w\) satisfies
\begin{eqnarray}
  \label{eq:8}
  \alpha_{1}\frac{\partial w}{\partial x_{1}} + Q(\xi)w = g(x_{1},\xi)
\end{eqnarray}
where
\(Q(\xi) =
-\frac{1}{2}\sum_{j=2}^{n}\epsilon_{j}\xi_{j}^{2} + i\sum_{j=2}^n \alpha_j\xi_j +\frac{1}{2}B\). We
may write an explicit formula for \(w\):
\begin{eqnarray}
  \label{eq:9}
  w(x_1,\xi) = \frac{1}{\alpha_{1}}
  \begin{cases}
    \int_{-\infty}^{x_{1}}e^{Q(\xi)\frac{x_1-s}{\alpha_{1}}}g(s,\xi)ds
    &\mathrm{if}\ \frac{\Re{Q(\xi)}}{\alpha_{1}}>0
\\
-\int_{x_{1}}^{\infty}e^{Q(\xi)\frac{x_1-s}{\alpha_{1}}}g(s,\xi)ds
    &\mathrm{if}\ \frac{\Re{Q(\xi)}}{\alpha_{1}}<0
  \end{cases}
\end{eqnarray}
Our formula insures that, on the domain of integration,
\begin{equation}\label{eq:10}
  |e^{Q(\xi)\frac{x-s}{\alpha_{1}}}|<1
\end{equation}
  and therefore, for each fixed  \(\xi\), that
\begin{equation}\nn
  ||w(\cdot,\xi)||_{L^\infty}\le
  \frac{1}{\alpha_{1}}||g(\cdot,\xi)||_{L^1}
\end{equation}
  so that, squaring and integrating with respect
  to  \(\xi\) gives
\begin{eqnarray}\label{eq:31}
 ||w||_{L^2(d\xi,L^\infty(dx_{1}))}\le
\frac{1}{\alpha_{1}}||g||_{L^2(d\xi,L^1(dx_{1}))}
\\\na{or, using the notation of mixed norms}\nn
||\ftpj1{u}||_{\Theta_{1}(\infty,1)}\le
\frac{1}{2\alpha_{1}}||\ftpj1{f}||_{\Theta_{1}(1,2)}
\end{eqnarray}
with  \(\Theta_{1}\) equal to the unit vector in the
\(x_{1}\) direction.   
This combines with Lemma \ref{mixedToUniformNorms} to yield the
estimate (\ref{eq:2}). 
\end{proof}
The proof of Theorem \ref{th:main} will use partitions of
unity and coordinate changes to reduce to a case very
similar to (\ref{eq:8}) and (\ref{eq:9}) and prove
estimates of the form in (\ref{eq:6}). Our main proof will
prove Theorem \ref{th:main} in the case that no
\(\epsilon_{j}\) in (\ref{eq:4}) is zero. We have already
treated the case where some \(\epsilon_{j}=0\) and the
corresponding \(\alpha_{j}\ne 0\). If, for one or more
values of \(j\), \(\epsilon_{j}=\alpha_{j}=0\), then the
PDE in (\ref{eq:1}) is independent of the \(x_{j}\)
variables. In this case, we may obtain the inequality
(\ref{eq:6}) from the corresponding inequality in the
lower dimensional case. We record this in the proposition
below.
\begin{proposition}\label{prop4.4}
 Let  \(x=(x_{1},\tilde{x},y)\), and suppose that, for
 each  \(y\),
\begin{equation}\label{eq:14}
   ||u(\cdot,\cdot,y)||_{L^2(d\tilde{x},L^\infty(dx_{1}))}\le
   C ||f(\cdot,\cdot,y)||_{L^2(d\tilde{x},L^1(dx_{1}))}
\end{equation}
  then
\begin{equation}
   ||u||_{L^2(d\tilde{x} dy,L^\infty(dx_{1}))}\le
   C ||f(\cdot,\cdot,y)||_{L^2(d\tilde{x} dy,L^1(dx_{1}))}
\end{equation}
\end{proposition}
\begin{proof}
  Just square both sides of (\ref{eq:14}) and integrate
  with respect to  \(y\). 
\end{proof}

Henceforth, we will assume that no  \(\epsilon_{j}=0\),
and complete the squares in (\ref{eq:4}) to rewrite that
equation as
\begin{eqnarray}
  \nn
  P(\D) = \sum_{j=1}^{n}\epsilon_{j}(\Dj-\beta_{j})^{2}
  + b\qquad;\quad \epsilon_{j}=\pm 1
\end{eqnarray}
where the  \(\beta_{j}=-\epsilon_{j}\alpha_{j}\) from
(\ref{eq:4}) and  \(b=B-\sum\epsilon_{j}\beta_{j}^{2}\) .  
\begin{proposition}
   \(P(\D)\) has a real double characteristic iff  \(b=0\)
   and  \(\beta=\overline{0}\).   
\end{proposition}
\begin{proof}
\begin{align}\label{eq:19}
        p(\eta) = \sum\epsilon_j(i\eta_{j}-\beta_{j})^{2}+b
\\ \notag
    dp = \sum 2i\epsilon_j(i\eta_{j}-\beta_{j})d\eta_{j}
\end{align}
  so that
\[
    dp = 0\quad \iff\quad \mathrm{every}\
    \eta_{j}=-i\beta_{j}
\]
  but, as the  \(\beta_{j}\) are real, this can
    only happen if
\[
  \eta_{j} = \beta_{j}=0 
\]
If \(p\) vanishes as well, we must also have  \(b=0\).  
\end{proof}

We now begin the proof of Theorem \ref{th:main} in
earnest. We intend to use partial Fourier transforms, as
defined
in (\ref{eq:68}). To this end, we will choose special
directions  \(\Theta\in\R^{n}\) (the unit vectors
\(\Theta_{k}\) in the coordinate directions
will suffice for the proof of Theorem \ref{th:main}) and
express  \(x\in\R^{n}\) as
\begin{eqnarray*}
  x = t\Theta + \tp{x} 
\end{eqnarray*}
as in \eqref{eq:83} 
and 
write the dual variable  \(\eta\) as
\begin{eqnarray*}
 \eta = \tau\Theta + \xi \qquad\mathrm{with\ }
  \xi\in\Theta^\perp
\end{eqnarray*}
In these coordinates,
we consider \(p(\eta)\)
as a polynomial  \(p(\tau;\xi)\) in \(\tau\)
with coefficients depending on \(\xi\).
We will arrive at the estimate \eqref{eq:2} as long as the
roots of \(p\) are simple. When \(\Theta=\Theta_{k}\), 
\(\xi=(\eta_{1},\ldots\eta_{k-1},\eta_{k+1},\ldots,\eta_{n})\). If
we define
\begin{equation}\label{eq:20}
  Q_{k}(\xi) := \sum_{j\ne k}\epsilon_j(i\eta_{j}-\beta_{j})^{2}+b
\end{equation}
  then by \eqref{eq:19} we have $p(\tau,\xi) = 
  \epsilon_k(i\tau-\beta_k)^2 + Q_k(\xi)$ and its roots are
\[
  \tau_{\pm} = -i\beta_{k}\pm \sqrt{\epsilon_kQ_{k}(\xi)}
\]
  and they are simple as long as
\[
  Q_{k}(\xi)\ne 0.
\]

\begin{proposition}\label{th:ThetaEst}
  Suppose that 
  \begin{eqnarray}
    \label{eq:17}
    \supp{\widehat{f}(\eta)}\subset  \{\eta\in\R^n \big|\ |Q_{k}(\xi)|>\varepsilon\}
  \end{eqnarray}
Then there
  exists  \(u\) solving
\begin{equation}\nn
  P(\D)u=f
\end{equation}
  satisfying
\begin{equation}\label{eq:22}
    ||\ftpk{u}||_{\Theta_{k}(\infty,2)}\le\frac{1}{\sqrt{\varepsilon}}
    ||\ftpk{f}||_{\Theta_{k}(1,2)}
\end{equation}
\end{proposition}
\begin{proof}
  With  \(x=t\Theta_{k}+\tp{x}\), we again use
the partial Fourier transform
\begin{eqnarray}
  \nn
  \ftpk{u}(t,\xi) = \frac{1}{(2\pi)^{(n-1)/2}}
\int u(t,\tpk{x})e^{-i\xi\cdot \tpk{x}}d\tpk{x}
\end{eqnarray}
Letting \(w:=\ftpk{u}(t,\xi)\) and  \(g=\ftp{f}(t,\xi)\),
we see that 
\begin{equation}\nn
  \epsilon_k \left(\frac{d}{dt}-\beta_{1}\right)^{2}w
  +Q_{k}(\xi)w = g
\end{equation}
  which factors as
\[
  \left(\frac{d}{dt}-(\beta_{1}+\sqrt{\epsilon_kQ_{k}})\right)
  \left(\frac{d}{dt}-(\beta_{1}-\sqrt{\epsilon_kQ_{k}})\right)w = 
  \epsilon_k g
\]
so that we can  write a solution
formula analogous to that in \eqref{eq:59} through
\eqref{eq:69}; i.e.
\begin{eqnarray}
  \nn
  w = \frac{\epsilon_k}{\sqrt{Q_{k}}}\left(\int
  e^{(\beta_{1}+\sqrt{\epsilon_kQ_{k}})(t-s)}g(s,\xi)ds -
  \int
  e^{(\beta_{1}-\sqrt{\epsilon_kQ_{k}})(t-s)}g(s,\xi)ds\right)
\end{eqnarray}
where the limits of integration in the first integral are
 \(-\infty<s<t\) for those  \(\xi\) that satisfy
 \(\Re{(\beta_{1}+\sqrt{\epsilon_kQ_{k}})}>0\) and    \(t<s<\infty\)
 for  \(\xi\) with
 \(\Re{(\beta_{1}+\sqrt{\epsilon_kQ_{k}})}<0\). The limits in the
 the second integral are chosen similarly, based on the
 real part of  \(\beta_{1}-\sqrt{\epsilon_kQ_{k}}\). We
 may choose either set of limits if the real part is
 zero.\\

We now obtain the estimate (\ref{eq:22}) just as in
(\ref{eq:10}) through (\ref{eq:31}).
\end{proof}

Our next step is to show that any compactly supported
\(f\in L^2\) can be decomposed into a sum of sources, each
of which will satisfy (\ref{eq:17}) for some
\(\Theta_{k}\). To accomplish this, we let  \(\phi(t)\in C^{\infty}_{0}(\R)\)
be a positive bump function, equal to 0 for  \(|t|<1\) and
1 for  \(|t|>2\). We let  \(\phi_{\varepsilon}(t) =
\phi(\frac{t}{\varepsilon})\).  Again writing  \(\eta\in\R^{n}\) as
\begin{eqnarray}
  \nn
  \eta=\tau\Theta_{k}+\xi
\end{eqnarray}
it is natural to define the multiplier
\begin{eqnarray}
  \nn
  \Phi_{k}(\eta) = \phi_{\varepsilon}(|Q_{k}(\xi)|)
\end{eqnarray}
which will equal 0 near the set where  \(Q_{k}\) is
small. It is, however, more convenient to define
\begin{eqnarray}
  \label{eq:32}
  \Phi_{k}(\eta) = \phi_{\varepsilon}(\Re{Q_{k}}) +
\phi_{\varepsilon}(\Im{Q_{k}}) - \phi_{\varepsilon}(\Re{Q_{k}})\phi_{\varepsilon}(\Im{Q_{k}})
\end{eqnarray}
which equals 0 if  both \(\Re{Q_{k}}<\varepsilon\) and 
  \(\Im{Q_{k}}<\varepsilon\), and equals 1 if either or both is greater
  than  \(2\varepsilon\). We decompose  \(f\) as
\begin{align}
    &\widehat{f} = \widehat{f_{1}} + \widehat{f_{2}}+\widehat{f_{3}}+\ldots \widehat{f_{n+1}}&
\notag\\ 
    &= \Phi_{1}\widehat{f} + \Phi_{2}(1-\Phi_{1})\widehat{f}+
\Phi_{3}(1-\Phi_{2})(1-\Phi_{1})\widehat{f}
+\ldots
      +\mathop{\prod}\limits_{j=1}^{n}(1-\Phi_{j})\widehat{f}
\label{eq:33}
\end{align}
  and solve
\begin{equation}\label{eq:63}
  P(\D)u_{k}=f_{k}
\end{equation}
  which will guarantee that, for all \(k=1\ldots n\),
  \(f_{k}\)
  will satisfy the hypothesis (\ref{eq:17}) of Proposition
  \ref{th:ThetaEst} with direction vector \(\Theta_{k}\). We will use that proposition to
  construct and estimate the \(u_{k}\).
  To estimate the solution to \(P(\D)u_{n+1}=f_{n+1}\)
  we will need the following:

\begin{lemma}\label{th:zqeps}
  Let 
  \begin{eqnarray}
    Z^{\varepsilon}_{Q_{k}} = \{\eta\in\R^n \big|\ 
    |\Re{Q_{k}(\xi)}|<\varepsilon\ \mathrm{and}\ |\Im{Q_{k}(\xi)}|<\varepsilon\}
  \end{eqnarray}
then  \(\displaystyle \mathop{\cap}\limits_{k=1}^{n}Z^{\varepsilon}_{Q_{k}}\) is bounded
with diameter less than  \(4\sqrt{2n\varepsilon}\).
Moreover, if P has no double characteristics, and
\(\varepsilon\) is chosen small enough, 
\begin{eqnarray}
  \label{eq:36}
  \mathop{\cap}\limits_{k=1}^{n}Z^{\varepsilon}_{Q_{k}}\cap Z^{\varepsilon}_{P}=\emptyset
\end{eqnarray}
where $Z^\varepsilon_{P}$ is defined similarly as $Z^\varepsilon_{Q_k}$.
\end{lemma}
Before we begin the proof we record one simple lemma,
which we will use here and again in the proof of Proposition
\ref{th:qEst}.
\begin{lemma}\label{th:simplequad}
  Suppose that  \(q(t)=t^{2}-B\) and
  \(Z^{q}_{\delta}=\{t\in\R \big| |q(t)|<\delta\}\), then
  \begin{eqnarray}
    \mu\left(Z^{q}_{\delta}\right)\le 4\min\left(\sqrt{\delta},\frac{\delta}{\sqrt{B}}\right)
  \end{eqnarray}
\end{lemma}
\begin{proof}
  If  \(B<-\delta\),  \(Z^{q}_{\delta}\) is empty, so
  assume that is not the case and $t\in Z^\varepsilon_\delta$. Then
  \begin{eqnarray}
    \nn
    \max(0,B-\delta)\leqslant t^{2}<B+\delta
  \end{eqnarray}
so  \(\pm t\) belongs to the interval
\(\left[\sqrt{\max(0,B-\delta)},\sqrt{B+\delta}\right]\), which
has length
\begin{eqnarray}
  \nn
  \sqrt{B+\delta} - \sqrt{\max(0,B-\delta)} &=&
\frac{2\delta}{\sqrt{B+\delta} + \sqrt{\max(0,B-\delta)}}
\\
&\le&\frac{2\delta}{\max(\sqrt{B},\sqrt{\delta})}
\end{eqnarray}
\end{proof}

\begin{proof}[Proof of Lemma \ref{th:zqeps}]
  If
  \(\eta\in\mathop{\cap}\limits_{k=1}^{n}Z^{\varepsilon}_{Q_{k}}\),
  we will show that, each coordinate, \(\eta_{m}\)
  belongs to the union of two intervals, with total length
  at most  \(4\sqrt{2\varepsilon}\), so that the diameter of the set
  is no more than  \(\sqrt{n}\) times \(4\sqrt{2\varepsilon}\). 
  For
  \(\eta\in\mathop{\cap}\limits_{k=1}^{n}Z^{\varepsilon}_{Q_{k}}\),
\[
    \left|\sum_{k\ne m}Q_{k}\right| \le (n-1)\varepsilon
\]
  and
\begin{align*}
  \sum_{k\ne m}Q_k &= \sum_{k\ne m} \left[ \sum_{j\ne k}
    \epsilon_j (i\eta_j - \beta_j)^2 + b \right]\\
  &= (n-2) \left[ \sum_{j\ne m} \epsilon_j (i\eta_j-\beta_j)^2 + b \right] 
  +(n-1) \epsilon_m (i\eta_m-\beta_m)^2 + b\\
  &= (n-2) Q_m +(n-1) \epsilon_m (i\eta_m-\beta_m)^2 + b,
\end{align*}
  so
\[
  |(i\eta_m-\beta_m)^2 \pm b/(n-1)| \le (n-1)\varepsilon + (n-2)\varepsilon < 2\varepsilon.
\]
The real part of  \((i\eta_{j}-\beta_{j})^{2}\pm b/(n-1)\) is
\(-\eta_{m}^{2}+B\) with  \(B = \beta_{m}^{2}\pm b/(n-1)\), so
we may invoke Lemma \ref{th:simplequad}  with
\(\delta=2\varepsilon\) to conclude that  \(\eta_{m}\) belongs to
set with diameter at most \(4\sqrt{2\varepsilon}\).\\

We perform a similar calculation to establish
(\ref{eq:36}). The absence of real double characteristics means
that either  \(b\) or some  \(\beta_{j}\) in
(\ref{eq:19}) is nonzero. 
For \(\eta\in\mathop{\cap}\limits_{k=1}^{n}Z^{\varepsilon}_{Q_{k}}\), 
\begin{eqnarray*}
  \left|\sum_{k=1}^{n}Q_{k}\right|&\le& n\varepsilon
\\
\left|(n-1)\sum_{k=1}^{n}\epsilon_{k}(i\eta_{k}-\beta_{k})^{2}+nb\right|&\le&
  n\varepsilon
\\
\left|(n-1)p(\eta)+b\right|&\le& n\varepsilon
\\
\left|p(\eta)\right|&\ge& \frac{|b|-n\varepsilon}{n-1}
\\
&\ge&\varepsilon
\end{eqnarray*}
as long as  \(|b|>0\) and  \(\varepsilon\) is chosen sufficiently smaller
than  \(|b|\) . If  \(b=0\), then some  \(\beta_{k}\ne0\) and
\begin{eqnarray*}
 p-Q_{k} = \epsilon_k(i\eta_{k}-\beta_{k})^{2}
\\
|p|\ge|i\eta_{k}-\beta_{k}|^{2}-|Q_{k}|
\\
|p|\ge \beta_{k}^{2}-\varepsilon
\\\ge\varepsilon
\end{eqnarray*}
for  \(\varepsilon\) sufficiently smaller than  \(\beta_{k}^{2}\). 
\end{proof}

Proposition \ref{th:ThetaEst} gives us the estimates
\begin{eqnarray}
  \nn
  ||\f{\Theta_k^\perp}{u_{k}}||_{\Theta_{k}(\infty,2)}\le\frac{1}{\sqrt{\varepsilon}}
    ||\f{\Theta_k^\perp}{f_{k}}||_{\Theta_{k}(1,2)}
\end{eqnarray}
for \(k=1\ldots n\). To estimate  \(u_{n+1}\), we prove

\begin{proposition}\label{th:unplus1}
  Suppose that  \(\supp \widehat{f}\) has  diameter at most
  \(d \) , and further that  \(|P(\eta)|>\varepsilon\) on
  \(\supp \widehat{f}\). Then  
\[
  u := \mathscr{F}^{-1}\left(\frac{\widehat{f}}{P}\right)
\qquad
  \mathrm{solves}
\qquad
  P(D)u = f
\]
  and, for any  unit vector \(\Theta\),
\begin{equation}
  ||\ftp{u}||_{\Theta(\infty,2)} \le \frac{d}{\varepsilon}||\ftp{f}||_{\Theta(1,2)}
\end{equation}
\end{proposition}
\begin{proof}  We write  \(\eta=\tau\Theta+\xi\), and  inverse
  Fourier transform in the  \(\Theta\) direction, obtaining
\begin{equation}\nn
  |\ftp{u}(t,\xi)| \le \left| \finv{\Theta}{\left\{ \frac{\chi_{\supp \widehat{f}}}{P} \right\}} 
  \ast_t \f{\Theta^\perp}{f} \right| \le \frac{d}{\varepsilon} \int | \f{\Theta^\perp}{f}(t',\xi)| dt'
\end{equation}
  so for each fixed  \(\xi\),
\[
 ||\ftp{u}(\cdot,\xi)||_{L^{\infty}}\le \frac{d}{\varepsilon}||\ftp{f}(\cdot,\xi)||_{L^1(dt)}
\]
  Taking  \(L^{2}(d\xi)\) norms of both sides yields
\[
 ||\ftp{u}||_{\Theta(\infty,2)} \le \frac{d}{\varepsilon}||\ftp{f}||_{\Theta(1,2)}
\]
\end{proof}

To complete the proof of Theorem \ref{th:main}, we need
only show that,  for \hbox{\(k=1\ldots n+1\)},
\begin{eqnarray}
  \label{eq:39}
  ||\ftpk{f_{k}}||_{\Theta_{k}(1,2)} \le C ||\ftpk{f}||_{\Theta_{k}(1,2)}
\end{eqnarray}
We will then apply Lemma \ref{mixedToUniformNorms} to conclude that
each  \(u_{k}\)  satisfies 
\begin{equation}\nn
  ||u_{k}||_{L^2(D_{r})} \le \sqrt{d_{r}}||\ftpk{u_{k}}||_{\Theta(\infty,2)}
\end{equation}
  and
  \begin{eqnarray}\nn
  ||\ftpk{f}||_{\Theta_{k}(1,2)} \le C
    \sqrt{d_{s}}||f||_{L^2(D_{s)}}\nn
  \end{eqnarray}
Recalling that  \(u=\sum u_{k}\) in
(\ref{eq:63}) will then finish the proof of Theorem
\ref{th:main}. \textit{Note that we can't apply Lemma \ref{mixedToUniformNorms}
  directly to the  \(f_{k}\) because their
  supports need not be contained in the support of  \(f\).}\\

In order to establish (\ref{eq:39}) for  \(f_{k}\) defined
as in  \eqref{eq:33}, we need to estimate  \(||\ftpk{
  M_{\Psi_{j}}f}||_{\Theta_{k}(1,2)}\) for all  \(j\) and
\(k\). The case  \(j=k\) is the simplest.

\begin{lemma}\label{th:jequalsk} Let  \(\Psi_{k}\) denote either \(\phi_{\varepsilon}(\Re{Q_{k}})\) or
   \(\phi_{\varepsilon}(\Im{Q_{k}})\). Then, 
  \begin{eqnarray}\label{eq:44}
   ||\ftpk{ M_{\Psi_{k}}f}||_{\Theta_{k}(1,2)}\le
    ||\ftpk{f}||_{\Theta_{k}(1,2)}
  \end{eqnarray}
\end{lemma}
\begin{proof}
   Recall that , writing  \(\eta= \tau\Theta_{k}+\xi\) with  \(\xi\in\Theta_{k}^{\perp}\),
  \begin{eqnarray*}
    Q_{k}(\eta)=Q_{k}(\tau\Theta_{k}+\xi)=Q_{k}(\xi)
  \end{eqnarray*}
so that  \(Q_{k}\), and therefore  \(\Psi_{k}\) does not
depend on  \(\tau\). Hence
\begin{eqnarray*}
\ftp{M_{\Psi_{k}}f} &=&
             \Psi_{k}(\xi)\ftp{f}(t,\xi)
\\
 ||\ftp{M_{\Psi_{k}}f}||_{\Theta_{k}}(1,2)&\le&
 ||\Psi_{k}(\xi)||_{L^{\infty}}\;||\ftp{f}||_{\Theta_{k}}(1,2)
\end{eqnarray*}
and (\ref{eq:44}) now follows on noting that
\(|\phi_{\varepsilon}|\le 1\). 
\end{proof}

According to Lemma \ref{multiplierNorm}, we may establish
  (\ref{eq:39}) for  \(j\ne k\) by proving that, 
  \(||\finv{\Theta_{k}}{\Psi_{j}}||_{\Theta_{k}(1,\infty)}\)
  is bounded. We address this in the next few lemmas.

\begin{lemma}\label{th:BasicEst}
  Let  \(q(t)\) be a real valued function of  \(t\in\R\),
  and let
  \begin{eqnarray}
  \nn
        \Phi(t)&=&\phi_{\varepsilon}(q(t))
\\
Z^{\varepsilon}_{q}&=&\{t\in\R \big|\ |q(t)|<\varepsilon\}
  \end{eqnarray}
Suppose that
\begin{eqnarray}
  \mu(Z^{\varepsilon}_{q}) \le \mu_{1}
\qquad
  \sup_{t\in
  Z^{\varepsilon}_{q}}\left|\frac{dq}{dt}\right|\le M_{1}
\qquad
  \sup_{t\in
  Z^{\varepsilon}_{q}}\left|\frac{d^{2}q}{dt^{2}}\right|\le M_{2}
\end{eqnarray}
  then
\begin{eqnarray}
  ||\widecheck{\Phi}||_{L^{1}} \le 2\mu_{1}
  \left[\frac{M_{1}}{\varepsilon}+\sqrt{\frac{M_{2}}{\varepsilon}}\right]
\end{eqnarray}
where  \(\widecheck{\Phi}\) denotes the (one dimensional)
inverse Fourier transform of  \(\Phi\). 
\end{lemma}
\begin{proof}
  \begin{eqnarray*}
    \left|\widecheck{\Phi}(\tau)\right| &=& 
    \left| \frac{1}{\sqrt{2\pi}} \int
    e^{-it\tau}\phi_{\varepsilon}(q(t))dt\right|
\le\mu(Z^{\varepsilon}_{q})\le\mu_{1}
  \end{eqnarray*}
Two integrations by parts yield
\begin{eqnarray*}
  \left|\widecheck{\Phi}(\tau)\right| &=& \left|\frac{-1}{\tau^{2}\sqrt{2\pi}}\int
    e^{it\tau}\left[\phi_{\varepsilon}'\;q''+
                                   \phi_{\varepsilon}''\;(q')^{2}\right]dt\right|
\le\frac{1}{\tau^{2}}
\left[\frac{M_{2}}{\varepsilon}+\left(\frac{M_{1}}{\varepsilon}\right)^{2}\right] 
\end{eqnarray*}
so that
\[
  \left|\widecheck{\Phi}(\tau)\right| \le \mu_{1}
\begin{cases}
   1, &\tau \le
   \left[\frac{M_{2}}{\varepsilon}+\left(\frac{M_{1}}{\varepsilon}\right)^{2}\right]^{\frac{1}{2}}
\\
\frac{\left[\frac{M_{2}}{\varepsilon}+\left(\frac{M_{1}}{\varepsilon}\right)^{2}\right]}{\tau^{2}},
& \tau \ge \left[\frac{M_{2}}{\varepsilon}+\left(\frac{M_{1}}{\varepsilon}\right)^{2}\right]^{\frac{1}{2}}
  \end{cases}
\]
  which implies that
\[
  \int\left|\widecheck{\Phi}(\tau)\right|d\tau
  \le
      2\mu_{1}\left[\frac{M_{2}}{\varepsilon}+\left(\frac{M_{1}}{\varepsilon}\right)^{2}\right]^{\frac{1}{2}}
\le 2\mu_{1}\left[\left(\frac{M_{2}}{\varepsilon}\right)^{\frac{1}{2}}+\frac{M_{1}}{\varepsilon}\right]
\]
\end{proof}
  An immediate corollary is:

\begin{corollary}\label{th:BasicEstCor}
  Let  \(Q(t,\xi)\) be a real valued function of \(t\in\R\), \(\xi\in\R^n\),
  and let
  \begin{eqnarray}
  \nn
        \Phi(t\Theta+\xi)&=&\phi_{\varepsilon}(Q(t,\xi))
\\
Z^{\varepsilon}_{Q}(\xi)&=&\{t\in\R\big|\ |Q(t,\xi)|<\varepsilon\}
  \end{eqnarray}
Suppose that
\begin{eqnarray}
    \mu(Z^{\varepsilon}_{Q})\le \mu_{1}(\xi)
\qquad
  \sup_{t\in
  Z^{\varepsilon}_{Q}}\left|\frac{dQ}{dt}\right|\le M_{1}(\xi)
\qquad
  \sup_{t\in
  Z^{\varepsilon}_{Q}}\left|\frac{d^{2}Q}{dt^{2}}\right|\le M_{2}(\xi)
\end{eqnarray}
  then
\begin{equation}\label{eq:54}
  ||\finv{\Theta}{\Phi}||_{\Theta(1,\infty)} \le \sup_{\xi}\mu_{1}(\xi)
  \left[\frac{M_{1}(\xi)}{\varepsilon}+\sqrt{\frac{M_{2}(\xi)}{\varepsilon}}\right]
\end{equation}
\end{corollary}

Finally, we specialize to  \(\Theta=\Theta_{j}\) and
estimate the quantities on the right hand side of  (\ref{eq:54}).

\begin{lemma}\label{th:qEst}
  Let  \(Q(t,\xi) = \Re{Q_{k}(t\Theta_{j}+\xi)}\) or  \(Q(t,\xi) =
    \Im{Q_{k}(t\Theta_{j}+\xi)}\), with
\begin{eqnarray}
  \mu(Z^{\varepsilon}_{Q})=:\mu_{1}(\xi)
\qquad
  \sup_{t\in
  Z^{\varepsilon}_{Q}}\left|\frac{dQ}{dt}\right|=:M_{1}(\xi)
\qquad
  \sup_{t\in
  Z^{\varepsilon}_{Q}}\left|\frac{d^{2}Q}{dt^{2}}\right|=:M_{2}(\xi)
\end{eqnarray}
then
\begin{eqnarray}
  \label{eq:43}
  \sup_{\xi}\mu_{1}(\xi)
\left[\frac{M_{1}(\xi)}{\varepsilon}+\sqrt{\frac{M_{2}(\xi)}{\varepsilon}}\right]\le 9\sqrt{2}
\end{eqnarray}
\end{lemma}
\begin{proof}

    We write \(\eta = \sigma\Theta_{k}+t\Theta_{j}+ \xi\)
    where \(\xi\) is orthogonal to both \(\Theta_{k}\) and
    \(\Theta_{j}\).  Recall from (\ref{eq:20}) that
    \(Q_{k}\) does not depend on \(\sigma\), and assume
    for convenience that  \(\epsilon_{j}=+1\); First let
\begin{equation}\nn
  Q = \Re{Q_{k}(t,\xi,\sigma)} = t^{2} - C(\xi)
\end{equation}
  where
\[
  C(\xi) =
  \sum_{l\ne j,k}\epsilon_{l}\left(\xi_{l}^{2}-\beta_{l}^{2}\right)-b+\beta_{j}^{2}
\]
so that we may conclude from Lemma \ref{th:simplequad}
that
\begin{align}\nn
  \mu_{1}(\xi)\le 4
  \min\left(\sqrt{\varepsilon},\frac{\varepsilon}{\sqrt{|C(\xi)|}}\right)
  \\ \notag
  \left|\frac{dQ}{dt}\right| = 2|t|\le 2\sqrt{C(\xi)+\varepsilon}
\end{align}
  and
\[
  \left|\frac{d^{2}Q}{dt^{2}}\right| = 2
\]
  so that
\[
  \mu_{1}\frac{M_{1}}{\varepsilon}\le
  8\min\left(\sqrt{1+\frac{C(\xi)}{\varepsilon}},
  \sqrt{1+\frac{\varepsilon}{C(\xi)}}\right)
  \le 8\sqrt{2}
\]
  and
\[
  \mu_{1}\sqrt{\frac{M_{2}}{\varepsilon}}\le\sqrt{2}
\]
We next treat the case  \(Q= \Im{Q_{k}} =
2\beta_{j}t+2\sum_{l\ne j,k}\epsilon_l\beta_{l}\xi_{l}\). In this case

\begin{align*}
  \mu_{1}=\frac{\varepsilon}{2|\beta_{j}|}
  \\
  \frac{dQ}{dt} = 2\beta_{j}
\end{align*}
  and
\[
  \frac{d^{2}Q}{dt^{2}} = 0
\]
  so that
\[
  \mu_{1}\frac{M_{1}}{\varepsilon}= 1
\]
  and
\[
  \mu_{1}\sqrt{\frac{M_{2}}{\varepsilon}} = 0
\]
\end{proof}

The combination of Lemma \ref{th:BasicEst}, Corollary
\ref{th:BasicEstCor}, and Lemma \ref{th:qEst} gives us the
hypothesis necessary to invoke Lemma \ref{multiplierNorm}
and conclude that 
\begin{corollary}
  For  \(j\ne k\),
  \(||\finv{\Theta_{k}}{\Psi_{j}}||_{\Theta_{k}(1,\infty)}\le
  18\)
\end{corollary}
\noindent and consequently that (\ref{eq:39}) holds with
\(C = 19^{n+1}\) -- because
we use products (with  \(n+1\) factors) of these multipliers and identity minus
these multipliers for our cutoffs. 

We can now finish the
\begin{proof}[Proof of Theorem \ref{th:main}]
  We have shown that  multiplication by \(\phi_{\varepsilon}(\Re{Q_{k}})\) and
\(\phi_{\varepsilon}(\Im{Q_{k}})\) preserve bounds on  \(||\f{\Theta_k^\perp}{f}||_{\Theta_k(1,2)}\). 
Hence let us start with
\begin{align*}
  u &= \sum_{k} u_{k}
  \\
  ||u||_{L^2(D_{r})}&\le \sum_{k}||u_{k}||_{L^2(D_{r})}
  \\
  &\le \sum_{k}||u_{k}||_{L^2(S_{1})}
\end{align*}
  where  \(S_{1}\) is a strip bounded by the two planes
   \(\Theta_{k}\cdot x = s_{1}\) and \(\Theta_{k}\cdot x =
   s_{2}\), with \(|s_{2}-s_{1}|\le d_{r}\)
\begin{align*}
  &= \sum_{k}||\ftpk{u_{k}}||_{\Theta_{k}(2,2)(S_{1})}
  \\
  &\le\sum_{k}d_{r}^{\frac{1}{2}}||\ftpk{u_{k}}||_{\Theta_k(\infty,2)(S_{1})}
  \\
  &\le C_{1}(P,n)\ d_{r}^{\frac{1}{2}}\sum_{k}||\ftpk{f_{k}}||_{\Theta_k(1,2)}
  \\
  &\le  C_{2}(P,n) \ d_{r}^{\frac{1}{2}}\sum_{k}||\ftpk{f}||_{{\Theta_k(1,2)}}
  \\
  &\le  C_{3}(P,n)\ d_{r}^{\frac{1}{2}}\sum_{k}||\chi_{S_2}\ftpk{f}||_{\Theta_k(1,2)}
\end{align*}
where the \(C_{i}\) are constants depending only on \(P\)
and the dimension \(n\), and \(S_{2}\) is a strip
containing \(D_{s}\) defined analogously to \(S_{1}\).
\begin{align*}
  &\le C_{3}(P,n)\ d_{r}^{\frac{1}{2}}\sum_{k}d_{s}^{\frac{1}{2}}||\chi_{S_2}\ftpk{f}||_{{\Theta_k(2,2)}}
  \\
  &\le C_{4}(P,n)\  (d_{r}d_{s})^{\frac{1}{2}}||f||_{L^2}
\end{align*}
and Theorem \ref{th:main} is proved.
\end{proof}

\section{Estimates for higher order operators}\label{sec:ordern}

In this section we will consider an
$N$\textsuperscript{th} order constant coefficient partial
differential operator $P(D)$ on $\R^n$, $D = -i\nabla$. We
refer to polynomials which satisfy the three conditions
of Definition \ref{admissibleSymbol} as \textit{admissible}. For these admissible
polynomials, we will prove the same estimate as we did for
second order operators in Theorem \ref{th:main}. We will again use
partial Fourier transforms and solve ordinary differential
equations, using partitions of unity to decompose our
source into a sum of sources, each of which has support
suited to that particular direction, so that the solution
to the ODE satisfies the same estimates as in the previous
section.\\

The main difference here is that we don't have a simple
normal form as we did in Lemma \ref{th:1}, so we cannot
explicitly choose directions and construct cutoffs. We
need to rely on algebraic properties of the discriminant
to guarantee that we can find a finite decomposition of
the source analogous to the one we used in
\eqref{eq:85}. Additionally, the order of the ODE can
depend on the direction. In the second order case we
dismissed these cases easily in propositions \ref{prop4.3}
and \ref{prop4.4} because we could represent them
explicitly. In the higher
order case, we choose our directions to avoid these cases.\\

\begin{theorem}\label{th:other}
  Let $P:\C^n\to\C$ be a degree $N \geqslant 1$ admissible polynomial. Then 
  there is a constant $C = C(P,n)$ such that for every bounded domain 
  $D_s \subset \R^n$ and every $f\in L^2(D_s)$ there is a $u\in 
  L^2_{loc}(\R^n)$ satisfying
\begin{equation}
  P(D) u = f\label{eq:95}
\end{equation}
  Moreover for any bounded domain $D_r \subset \R^n$
\begin{equation}
  \norm{u}_{L^2(D_r)} \le C \sqrt{d_r d_s} \norm{f}_{L^2(D_s)}
\end{equation}
  where $d_\ell$ is the diameter of $D_\ell$.
\end{theorem}

\bigskip We will prove Theorem \ref{th:other} by reducing
the solution of the equation \eqref{eq:95} to solving a set of
parameterized ODE's, just as we wrote the solution to
\eqref{eq:13} in terms of solutions to
\eqref{modelProblems}. To accomplish this, we must choose
a set of directions $\Theta_k$ and build a partition of
unity on the Fourier side so that the denominators of the
source terms in each of these model problems are
strictly positive, just as was explained after
\eqref{eq:61}. These two ingredients will then imply the
final estimate.\\

  We choose a direction $\Theta\in\mathbb{S}^{n-1}(\R^n)$
  and Fourier transform \eqref{eq:95} along the $\Theta^\perp$ hyperplane to obtain
  the \emph{ordinary differential equation}
\[
  P(-i(\Theta\cdot\nabla)\Theta + \xi_{\Theta^\perp}) 
  \f{\Theta^\perp}{u} = \f{\Theta^\perp}{f}
\]
  in the direction $\Theta$, which we  solve for each 
  $\xi_{\Theta^\perp}$. The next lemma gives the estimate
  we seek in the case that ODE is first order.

 \begin{lemma}\label{PDEmodel}
   Let $\Theta\in\mathbb{S}^{n-1}$ and let $q:\Theta^\perp \to \C$ be 
   measurable. Assume that $g \in \Theta(1,2)$, or that $g \in 
   \Theta(\infty,2)$ and $\inf_{\Theta^\perp} \abs{\Im q} > 0$. Then 
   there is $w\in\Theta(\infty,2)$ satisfying $(-i\partial_t - 
   q(\xi_{\Theta^\perp})) w = g$ and
   \begin{equation}\label{eq:45}
   \norm{ w }_{\Theta(\infty,2)} \le \norm{ g}_{\Theta(1,2)},
    \end{equation}
   or in the second case
   \begin{equation}\label{eq:46}
   \norm{ w }_{\Theta(\infty,2)} \le \frac{1}{\inf_{\Theta^\perp} 
   \abs{\Im q}} \norm{ g }_{\Theta(\infty,2)}
 \end{equation}
\end{lemma}

 \begin{remark}
   It is also true that
   $\norm{w}_{\Theta(\infty,p)} \le
   \norm{g}_{\Theta(1,p)}$ and
   $\norm{w}_{\Theta(\infty,p)} \le
   \norm{g}_{\Theta(\infty,p)} /(\inf \abs{\Im q})$ with
   $1\le p \le \infty$.
 \end{remark}

 \begin{proof}
   The general solution to $(-i\partial_t - q(\xi_{\Theta^\perp})) w = g$ 
   is
 \[
   w(t\Theta+\xi_{\Theta^\perp}) = i\int_{t_0}^t 
   e^{iq(\xi_{\Theta^\perp})(t-t')} g(t'\Theta+\xi_{\Theta^\perp}) dt', 
   \quad t_0\in\R\cup\{\pm\infty\}.
 \]
   If $\Im q(\xi_{\Theta^\perp})=0$ we may set $t_0$ as we please and the 
   claim follows. If $\Im q < 0$ set $t_0 = \infty$, and then 
   $\abs{\exp(iq(\xi_{\Theta^\perp})(t-t'))} \le 1$ for 
   $t'\in{[{t,t_0}]}$. Similarly, if $\Im q > 0$ set $t_0 = -\infty$.

   Now \eqref{eq:45} follows by estimating the integral on
   the right by the  \(L^{\infty}\) norm of the
   exponential times the  \(L^{1}\) norm of  \(g\) for
   each fixed  \(\xi_{\Theta^\perp}\), and then taking
     \(L^{2}\) norms in the  \(\Theta^{\perp}\)
     hyperplane. The inequality   \eqref{eq:46} follows
     in the same way, but using the \(L^{1}\) norm of the
   exponential times the  \(L^{\infty}\) norm of  \(g\)
   instead of the other way around.

 \end{proof}

  In general the differential equation 
  $P(-i(\Theta\cdot\nabla)\Theta + \xi_{\Theta^\perp}) 
  \f{\Theta^\perp}{u} = \f{\Theta^\perp}{f}$ is not first 
  order in $\partial_t = \Theta\cdot\nabla$, but we can
  factor it into a product of first order operators  of the
  form \((-i\partial_t - q(\xi_{\Theta^\perp}))\), and then
  use a partial fractions expansion to express its
  solution as a sum of solutions to first order ODE's.

\begin{definition}
  For $\xi_{\Theta^\perp}$ fixed, let $p:\C\to\C$ be the polynomial in 
  $\tau$
\[
  p(\tau) = P(\tau\Theta+\xi_{\Theta^\perp}).
\]
  Then $p'(\tau) = \Theta\cdot\nabla P(\tau\Theta+\xi_{\Theta^\perp})$.
\end{definition}

\begin{lemma}\label{partialFractionDecomposition}
  Let $p:\C\to\C$ be a polynomial of degree $N \geqslant 1$. Assume that its 
  roots $\tau_j$ are simple and that its leading coefficient is $p_N$. 
  Then
\begin{equation}
  \frac{1}{p(\tau)} = \sum_{j=1}^N \frac{1}{(\tau-\tau_j) 
  p_N\prod_{k\neq j}(\tau_j-\tau_k)} = \sum_{j=1}^N 
  \frac{1}{(\tau-\tau_j) p'(\tau_j)}.
\end{equation}
\end{lemma}

\begin{proof}
  $p'(\tau_j) = \lim_{\tau\to\tau_j} p(\tau)/(\tau-\tau_j)$ since 
  $p(\tau_j)=0$.
\end{proof}

If, for some direction  \(\Theta_{k}\),
\(|p'(\tau_{j}(\xi_{\Theta_{k}^{\perp}}))|>\varepsilon>0\) for
all  \(\xi_{\Theta_{k}^{\perp}}\) and for all  \(j\) in the
support of  \(\ftpk{f_{k}}\) , then we can define  \(u_{kj}\) as
 solutions to  
\[
  \big(-i\partial_t - \tau_j(\xi_{\Theta_k^\perp})\big) 
  \f{\Theta_k^\perp}{u_{kj}} = \f{\Theta_k^\perp}{f_k}/p'(\tau_j)
\]
 that satisfy \eqref{eq:45} with
 \(w=\f{\Theta_k^\perp}{u_{kj}}\) and
 \(g=\ftpk{f_{k}}/p'(\tau_{j})\). Then 
 $u_k = \sum_j u_{kj}$ will solve $P(D)u_k 
  = f_k$.

  We must find a finite set of directions  $\Theta_k$, and split $f = 
  \sum_k f_k$ such that $\f{}{f_k}(\xi) = 0$ whenever 
  $\xi_{\Theta^\perp}$ is such that $|p'(\tau_j)|\le\varepsilon$ for any $j$, 
  as was done in \eqref{eq:33}. Thus we need to define the sets where
  $p'(\tau_j)$ becomes small. In reading the definition
  below, recall that  \(\xi_{\Theta^{\perp}}\) is the
  component of  \(\xi\) perpendicular to  \(\Theta\) and
  \(\tau_{j}=\tau_{j}(\xi_{\Theta^{\perp}})\), \(j=1,\ldots,N\) are the
  roots of  \(p(\tau)\).

\begin{definition}
  Given $\Theta\in\mathbb{S}^{n-1}$ and $\varepsilon \geqslant 0$ let
\begin{align}\label{actualBadCylinderDef}
  \mathscr{B}_{\Theta,\varepsilon} &= \{ \xi\in\R^n \mid 
  \min_{P(\tau_j\Theta+\xi_{\Theta^\perp})=0} \abs{\Theta\cdot\nabla 
  P(\tau_j\Theta+\xi_{\Theta^\perp})} \le \varepsilon \} \notag \\ &= 
  \{ \xi\in\R^n \mid \min_{p(\tau_j)=0} \abs{p'(\tau_j)} \le 
  \varepsilon \}
\end{align}
  where the minimum is taken with respect to
  $\tau_j\in\C$. If, for some  \(\Theta\),
  $\deg p = 0$ we adopt the convention that 
  $\min_{p(\tau_j)=0} \abs{p'(\tau_j)} := 0$.
\end{definition}

\begin{proposition}\label{solExistence}
  Let $P:\C^n\to\C$ be a polynomial of degree $N \geqslant 1$ with principal 
  term $P_N$. Assume that $P_N(\Theta)\neq0$. Let 
  $\f{\Theta^\perp}{f}\in\Theta(1,2)$ be such that $\f{}{f}(\xi) = 0$ 
  for all $\xi\in\mathscr{B}_{\Theta,\varepsilon}$. Then
  there exists  \(u\)  solving $P(D)u = f$ and
\begin{equation}
  \norm{ \f{\Theta^\perp}{u} }_{\Theta(\infty,2)} \le 
  \frac{N}{\varepsilon} \norm{ \f{\Theta^\perp}{f} }_{\Theta(1,2)}.
\end{equation}
\end{proposition}

\begin{remark}
  The mixed norm estimate is also true for any
  $1\le p\le\infty$:
  $\norm{\f{\Theta^\perp}{u}}_{\Theta(\infty,p)} \le N
  \varepsilon^{-1}
  \norm{\f{\Theta^\perp}{f}}_{\Theta(1,p)}$.
\end{remark}

\begin{proof}
  The roots of $p$ are simple when $\xi_{\Theta^\perp} \in 
  \R^n\setminus\mathscr{B}_{\Theta,\varepsilon}$ so we have
\[
  \frac{1}{p(\tau)} = \sum_{p(\tau_j)=0} 
  \frac{1}{(\tau-\tau_j)p'(\tau_j)}
\]
  there according to Lemma \ref{partialFractionDecomposition}. For each 
  $\xi_{\Theta^\perp}$, order the roots $\tau_j$ lexicographically by 
  $j=1,\ldots,N$, i.e. $\Re \tau_j \le \Re \tau_{j+1}$ and $\Im \tau_j 
  < \Im \tau_{j+1}$ if the real parts are equal. The maps 
  $\xi_{\Theta^\perp} \mapsto \tau_j(\xi_{\Theta^\perp})$ are measurable 
  since the coefficients of $p(\tau)$ are polynomials in the
  $\xi_{\Theta^\perp}$.
  
  The assumption on $\f{}{f}$ implies that
\[
  \norm{ \f{\Theta^\perp}{f}/p'(\tau_j) }_{\Theta(1,2)} \le 
  \varepsilon^{-1} \norm{ \f{\Theta^\perp}{f} }_{\Theta(1,2)} < \infty.
\]
  for any root $\tau_j = \tau_j(\xi_{\Theta^\perp})$ of $p(\tau_j)=0$.  
  Let $u_j \in \Theta(\infty,2)$ be the solution to
\[
  \big(-i\partial_t - \tau_j(\xi_{\Theta^\perp})\big) 
  \f{\Theta^\perp}{u_j} = \frac{ \f{\Theta^\perp}{f}} {p'(\tau_j)}
\]
  given by Lemma \ref{PDEmodel}. It satisfies the norm estimate
\[  
  \norm{\f{\Theta^\perp}{u_j}}_{\Theta(\infty,2)} \le 
  \varepsilon^{-1} \norm{ \f{\Theta^\perp}{f} }_{\Theta(1,2)}.
\]
  The claim follows by setting $u = \sum_{j=1}^N u_j$ and recalling the 
  partial fraction decomposition of Lemma 
  \ref{partialFractionDecomposition}.
\end{proof}

  We now focus on the second task, splitting an arbitrary 
  source function $f$ into a sum $f = f_1 + f_2 + \ldots + f_m$ with 
  directions $\Theta=\Theta_1, \Theta_2, \ldots, \Theta_m$ such that 
  $\f{}{f_k}(\xi) = 0$ when $\xi \in \mathscr{B}_{\Theta_k, 
  \varepsilon}$. Proposition \ref{solExistence} would then imply the 
  existence of a solution to $P(D)u_k=f_k$. Linearity then implies that 
  $u = u_1 + u_2 + \ldots + u_m$ solves the original
  problem $P(D)u=f$.\\

  The partial fraction expansion in Lemma
  \ref{partialFractionDecomposition} cannot hold if $P(D)$
  has a double characteristic, even a complex double
  characteristic. Unlike in Theorem \ref{th:main}, the
  algebraic techniques we use here rely on properties of
  the discriminant which involves the multiplicities of
  all the roots, including the the complex ones. Hence we
  require that \(P\) be what algebraic geometers call a
  \emph{nonsingular} polynomial.

\begin{definition}
  A polynomial $P:\C^n\to\C$ is \emph{nonsingular} if, given 
  $\xi\in\C^n$, $P(\xi) = 0$ implies that $\abs{\nabla P(\xi)}\neq 0$.
\end{definition}
  
The sets $\mathscr{B}_{\Theta, \varepsilon}$ are difficult
to deal with for a general polynomial $P$, but the sets
$\mathscr{B}_{\Theta,0}$ are algebraic sets, and
this will enable us to prove that the intersection of
finitely many of them is empty. In order to conclude that
the intersections of the
$\mathscr{B}_{\Theta, \varepsilon}$ are empty, we will
assume that each $\mathscr{B}_{\Theta, \varepsilon}$ is
contained in a tubular neighborhood of $\mathscr{B}_{\Theta,0}$.

Additionally, we require a compactness hypothesis on 
projections of two of the sets $\mathscr{B}_{\Theta,0}$
to insure that the cut-off function $\Psi$ associated with
$\mathscr{B}_{\Theta, \varepsilon}$ is a  Fourier
multiplier as in Lemma \ref{multiplierNorm}. 
  
\begin{definition}\label{admissibleSymbol}
  Let $P:\C^n\to\C$ be a degree $N \geqslant 1$ nonsingular polynomial with 
  principal term $P_N$. It is \emph{admissible} if
\begin{enumerate}
  \item\label{twoCylindersCompatibility} for any 
  $\Theta\in\mathbb{S}^{n-1}(\R^n) \setminus P_N^{-1}(0)$ and $r_0>0$ 
  there is $\varepsilon>0$ such that
\[
  \mathscr{B}_{\Theta, \varepsilon} \subset 
  \overline{B}(\mathscr{B}_{\Theta,0}, r_0),
\]  

  \item\label{compactDirections} there are non-parallel vectors 
  $\Theta_1, \Theta_2 \in \mathbb{S}^{n-1}(\R^n) \setminus P_N^{-1}(0)$ 
  such that $\mathscr{B}_{\Theta_1,0} \cap (\Theta_1)^\perp$ and 
  $\mathscr{B}_{\Theta_2,0} \cap (\Theta_2)^\perp$ are compact,
\end{enumerate}
  where $\mathscr{B}_{\Theta,\varepsilon}$ is defined in 
  \eqref{actualBadCylinderDef}.
\end{definition}

We suspect that Condition \ref{twoCylindersCompatibility} is true for 
any nonsingular polynomial. It has been straightforward to verify in the 
examples we have considered. Another way to state this condition is as 
follows: let $\mathscr{D}$ be the set of $\xi\in\Theta^\perp$ where 
$p(\tau)$, whose coefficients are polynomials of $\xi$, has a double 
root, i.e. $p(\tau_0)=p'(\tau_0)=0$. Let $r_0 > 0$. Then we require that 
there is $\varepsilon>0$ such that if $\xi\in\Theta^\perp$, 
$d(\xi,\mathscr{D}) > r_0$, then $\abs{p'(\tau_0)} > \varepsilon$ for 
all roots $\tau=\tau_0$ of $p(\tau)=0$.

Condition \ref{compactDirections} is likely
only a technical requirement.  Requiring it could be
avoided if a theorem similar to Corollary
\ref{th:BasicEstCor} and Lemma \ref{th:qEst} could be
proven for higher order operators. Moreover this condition
is always satisfied in $\R^2$ because each
$\mathscr{B}_{\Theta,0}$ is a finite set of lines in the
direction $\Theta$ in this case.

\bigskip A key point in our proof is the observation that
$\mathscr{B}_{\Theta,0}$ is an algebraic variety which can
be defined by the vanishing of a certain discriminant. We
first show that there is an infinite sequence of
directions $\Theta_k$ such that the intersection
$\cap_k \mathscr{B}_{\Theta_k,0}$ is empty. Because the
$\mathscr{B}_{\Theta_k,0}$ are algebraic varieties,
Hilbert's basis theorem then guarantees that the
intersection of a finite subset of the
$\mathscr{B}_{\Theta_k,0}$ is empty.

\begin{definition}\label{discDef}
  Let $P:\C^n\to\C$ be a polynomial of degree $N \geqslant 1$. Write $P_N$ for 
  its principal term. For any $\xi\in\C^n$ and $\Theta\in\C^n$ such that 
  $P_N(\Theta)\neq0$ we define
\begin{equation}
  \Delta(\Theta,\xi) = \operatorname{disc}_\tau(P(\tau\Theta+\xi)) := 
  \big(P_N(\Theta)\big)^{2(N-1)} \prod_{i<j} (\tau_i-\tau_j)^2,
\end{equation}
  where $\{\tau_j(\Theta,\xi) \mid j=1,\ldots,N\}$ are the roots of 
  $P(\tau\Theta+\xi)=0$. If $N = 1$ we set 
  $\operatorname{disc}_{\tau}(a_1\tau+a_0) = a_1$.
\end{definition}

\begin{remark}
  The discriminant of a polynomial \(P\) is a polynomial in the 
  coefficients of \(P\).  Hence we can extend $\Delta$ to the set 
  $\C^n\times\C^n$ by analytic continuation, and therefore it is 
  well-defined without the assumption that $P_N(\Theta)\neq0$.  We point 
  out, however, that the discriminant of a degree $N$ polynomial, with 
  the high-order coefficients equal to zero, is not the same as the 
  discriminant of the resulting lower degree polynomial. See for 
  example the introduction of Gel'fand, Kapranov and Zelevinsky 
  \cite{GKZ}.
\end{remark}

\begin{remark}
  We have $\Delta(\Theta,\xi) = \Delta(\Theta, \xi+r\Theta)$ for any 
  $r\in\C$. This follows from the fact that the roots of $P(\tau\Theta+ \xi+r\Theta) = 
  P((\tau+r)\Theta+\xi)$ are just the roots of
  $P(\tau\Theta+ \xi)$, all translated by  \(r\), so the
  discriminant remains the same. 
\end{remark}

\begin{remark}\label{th:DeltaHomog}
  We have $\Delta(\lambda\Theta,\xi) = \lambda^{N(N-1)} 
  \Delta(\Theta,\xi)$ because $P(\tau\lambda\Theta+\xi)$ has roots 
  $\tau_j = r_j/\lambda$ where $P(r_j\Theta+\xi)=0$, and the principal 
  term will be $\big(\lambda^N P_N(\Theta)\big)\tau^N$.
\end{remark}

\begin{definition}\label{algebraicBadCylinderDef}
  Let $\Theta\in\C^{n}$ and $P:\C^n\to\C$ be a degree $N \geqslant 1$ 
  polynomial. Then the \emph{algebraic} tangent set (in the direction 
  $\Theta$) is defined as
\begin{equation}
  \overline{\mathscr{D}_\Theta} = \{ \xi\in\C^n \mid 
  \Delta(\Theta,\xi) = 0\}.
\end{equation}
  The \emph{real} tangent set is $\mathscr{D}_\Theta = 
  \overline{\mathscr{D}_\Theta} \cap \R^n$.
\end{definition}

  Figure \ref{partitionUnityProcedureFig} on page 
  \pageref{partitionUnityProcedureFig} illustrates the
  example  $P(\xi) = 
  \abs{\xi}^2 - 1$ with $\Theta \in \{e_1, e_2, e_3\}$. We have then
\[
  P(\tau\Theta+\xi) = (\Theta\cdot\Theta) \tau^2 + 2(\Theta\cdot\xi) 
  \tau + \xi\cdot\xi-1
\]
  and 
\[
  \Delta(\Theta,\xi) = (\Theta\cdot\xi)^2 - 
  (\Theta\cdot\Theta)(\xi\cdot\xi-1).
\]
  Homogeneity is easy to see in this example and a simple calculation 
  demonstrates that $\Delta(\Theta,\xi+r\Theta) = \Delta(\Theta,\xi)$ 
  for all $r\in\C$, as expected.

  \bigskip We can study the sets $\mathscr{D}_\Theta$ as a proxy for the 
  sets $\mathscr{B}_{\Theta,\varepsilon}$, defined in 
  \eqref{actualBadCylinderDef}, that are actually used.

\begin{lemma}\label{DasBproxy}
  Let $P:\C^n\to\C$ be a degree $N \geqslant 1$ polynomial and 
  $\Theta\in\mathbb{S}^{n-1}(\R^n)$ such that $P_N(\Theta)\neq0$ . Then
\[
  \mathscr{D}_\Theta = \{ \xi\in\R^n \mid \exists \tau_0\in\C: p(\tau_0) 
  = p'(\tau_0) = 0 \} = \mathscr{B}_{\Theta,0}.
\]
\end{lemma}

\begin{proof}
  If $N \geqslant 2$ this follows from the definition of 
  $\mathscr{B}_{\Theta,0}$ in \eqref{actualBadCylinderDef} and the fact 
  that $\tau_0$ is a double root of $p$ if and only if $p(\tau_0)=0$ and 
  $p'(\tau_0)=0$. If $N=1$ then $\xi\in\mathscr{D}_\Theta$ iff the first 
  order coefficient of $p$ vanishes, which is the same condition as 
  $\xi\in\mathscr{B}_{\Theta,0}$. This is impossible since 
  $P_N(\Theta)\neq0$.
\end{proof}

  We will show that if $P$ is nonsingular then the intersection 
  $\cap_{\Theta \in \mathbb{S}^{n-1}} \mathscr{D}_{\Theta}$ is
  empty. In other words, we show that, given any $\xi\in\R^n$, there is 
  some direction $\Theta$ such that the line $\tau\mapsto 
  \tau\Theta+\xi$ is not tangent to the characteristic manifold 
  $P^{-1}(0)$ at any point.

\begin{lemma}\label{goodDirection}
  Assume that $P:\C^n\to\C$ is nonsingular. Let $\xi \in \C^n$. Then 
  there is $\Theta \in \mathbb{S}^{n-1}(\R^n)$ such that 
  $P_N(\Theta)\neq0$ and $\Delta(\Theta,\xi) \neq 0$.
\end{lemma}

\begin{proof}
  We keep the second variable $\xi$ fixed in this proof,
  and suppress the dependence on  \(\xi\), writing 
  $\Delta(\Theta) = \Delta(\Theta, \xi)$. We view
  \(P(\tau\Theta+\xi)\) as a polynomial
  \(p(\tau,\Theta)\) in  \(\tau\) and  \(\Theta\).  

  According to \cite{Hormander2}, Appendix 1.2., $\Delta$
  is a polynomial in $\Theta\in\C$ and
  $\Delta \not\equiv 0$ if \(p(\tau,\Theta)\)
  is square-free. A nontrivial complex polynomial cannot
  vanish identically on $\R^n$, and thus neither on
  $\R^n \setminus P_N^{-1}(0)$.

Hence, if
  $p(\tau,\Theta)$ has no square
  factor, there is a $\Theta\in\R^n$ such that
  $P_N(\Theta)\neq0$ and $\Delta(\Theta) \neq 0$. Because
  \(\Delta\), as pointed out in Remark
  \ref{th:DeltaHomog}, is a homogeneous function of
  \(\Theta\), we may scale  \(\Theta\) so it has unit
  length, and the lemma follows in this case.

  Next, we show that if \(p(\tau,\Theta)\) has a square
  factor, then $P(z)$, viewed as a polynomial of
  \(z\in\C^{n}\) has a square factor, which contradicts
  the assumption that $P$ is nonsingular. Suppose that
  \(p(\tau,\Theta) = \big(S_1(\tau,\Theta)\big)^2
  S_2(\tau,\Theta)\). If we choose \(\tau=\lambda\) and
  \(\Theta=(z-\xi)/\lambda\), then, for any  \(z\in\C\), 

\begin{eqnarray*}
  P(z)&=& P\left(\lambda \frac{z-\xi}{\lambda} +
          \xi\right)
  \\
       &=& p(\lambda,(z-\xi)/\lambda)
  \\
       &=& \big( S_1(\lambda, (z-\xi)/\lambda) \big)^2
           S_2(\lambda, (z-\xi)/\lambda )
  \end{eqnarray*}
so that, unless  \(S_{1}(\lambda,\Theta)\) is independent
of  \(\Theta\),  \(P(z)\) must have a square factor, which
is  a contradiction.  Suppose now that \(S_{1}\) is independent of  \(\Theta\).  
  It is a non-constant polynomial, so there is $\tau_0\in\C$ such that 
  $S_1(\tau_0) = 0$. If $\tau_0\neq0$, choosing
  $\lambda=\tau_0$ implies that $P\equiv0$. If $\tau_0=0$, then  
  $P(\tau\Theta+\xi)$ vanishes to at least second order at the point $\xi$ in every 
  direction $\Theta\in\C^n$. This means that $\xi$ is  a singular 
  point of $P$, again contradicting the hypothesis that
  \(P\) is nonsingular. Hence $p(\tau,\Theta)$ has no 
  square factors and thus $\Delta$ is not identically zero.
\end{proof}

\begin{proposition}\label{algProp}
  Let $P:\C^n\to\C$ be a nonsingular polynomial of degree $N \geqslant 1$ with 
  principal term $P_N$. Then there is a finite set of directions 
  $\Theta_1, \ldots, \Theta_m \in \mathbb{S}^{n-1}(\R^n) \setminus 
  P_N^{-1}(0)$ such that
\[
  \bigcap_{k=1}^m \overline{\mathscr{D}_{\Theta_k}} = \emptyset.
\]
\end{proposition}

\begin{proof}
  We recall a few facts from algebra. A ring $R$ is \emph{Noetherian} if 
  every ideal is finitely generated. Another characterization 
  is that every increasing sequence of ideals stabilizes at a 
  finite index. In other words, if $I_1 \subset I_2 \subset \ldots$ are 
  ideals in $R$, then there is $m < \infty$ such that $I_\ell = I_m$ for 
  all $\ell \geqslant m$.

  The ring of complex numbers is Noetherian: its only
  ideals are $\{0\}$ and $\C$. Hilbert's basis theorem
  says that polynomial rings over Noetherian rings are
  also Noetherian. If $V \subset \C^n$ is an affine
  variety then $V = \mathbb{V}(\mathbb{I}(V))$, where
\begin{align*}
  &\mathbb{V}(I) = \{ \xi \in \C^n \mid f(\xi) = 0 \quad \forall f \in 
  I\},\\ &\mathbb{I}(V) = \{ f \in \C[\xi_1, \ldots, \xi_n] \mid f(\xi) 
  = 0 \quad \forall \xi \in V\}.
\end{align*}

  \smallskip Now we begin the proof. Let $\Theta_1, \Theta_2, \ldots \in 
  \mathbb{S}^{n-1}(\R^n) \setminus P_N^{-1}(0)$ be a sequence that's 
  dense in the surface measure inherited from the Lebesgue measure of 
  $\R^n$. Set
\[
  V_\ell := \{\xi \in \C^n \mid \Delta(\Theta, \xi) = 0, \text{ for } 
  \Theta = \Theta_1, \Theta_2, \ldots, \Theta_\ell\} = \bigcap_{k=1}^\ell 
  \overline{\mathscr{D}_{\Theta_k}}.
\]
  We have $V_1 \supset V_2 \supset V_3 \supset \ldots$ and hence 
  $\mathbb{I}(V_1) \subset \mathbb{I}(V_2) \subset \mathbb{I}(V_3) 
  \subset \ldots$ etc. By Hilbert's basis theorem there is a finite $m$ 
  such that $\mathbb{I}(V_\ell) = \mathbb{I}(V_m)$ for all $\ell \geqslant 
  m$. This implies that $V_\ell = \mathbb{V}(\mathbb{I}(V_\ell)) = 
  \mathbb{V}(\mathbb{I}(V_m)) = V_m$ for $\ell\geqslant m$.

  If $V_m = \emptyset$ we are done. If not, then there is $\xi_{*}
  \in V_m$, such that 
\[
  \Delta(\Theta_k,\xi_{*}) = 0
\]
for all $k\in\N$. Because $\{\Theta_k\}$ is dense in
$\mathbb{S}^{n-1}(\R^n)$ and the discriminant is a
continuous function, we see that \(\Delta(\Theta, \xi_{*})=0\)
for all
$\Theta\in\mathbb{S}^{n-1}(\R^n)$, which contradicts Lemma
\ref{goodDirection}.
\end{proof}

\begin{proposition}\label{epsNeighborhood}
  Let $P:\C^n\to\C$ be a nonsingular polynomial of degree $N \geqslant 1$. Let 
  $\Theta_k\in\mathbb{S}^{n-1}(\R^n)$ be a finite sequence of 
  non-parallel vectors such that $\cap_k \mathscr{D}_{\Theta_k} = 
  \emptyset$.
  
  If $\mathscr{D}_{\Theta_k} \cap \Theta_k^\perp$ is compact for 
  $k=1,2$ then there is $r_0>0$ such that
  \begin{eqnarray}\label{eq:89}
  \bigcap_k \overline{B}(\mathscr{D}_{\Theta_k}, 2r_0) =
    \emptyset.
  \end{eqnarray}
  Moreover, there are smooth $\Psi_k:\R^n\to{[{0,1}]}$
  such that $\Psi_k$ are bounded Fourier multipliers
  acting on $\finv{\Theta^\perp}{\Theta(1,2)}$ for every
   $\Theta\in\mathbb{S}^{n-1}(\R^n)$, satisfying
\begin{equation}
  \sum_k \Psi_k \equiv 1
\end{equation}
  and $\Psi_k \equiv 0$ in $B(\mathscr{D}_{\Theta_k}, r_0)$.
\end{proposition}

\begin{proof}
  If \(\mathscr{D}_{\Theta_1}\) is empty, then so is any
  neighborhood of it, hence the intersection in
  \eqref{eq:89} is empty. If 
  not, there are at least two linearly independent $\Theta_k$. Then the 
  intersection $\mathscr{D}_{\Theta_1} \cap \mathscr{D}_{\Theta_2}$ is 
  compact because our assumption that the first two
  $\mathscr{D}_{\Theta_k} \cap \Theta_k^\perp$ are compact
  implies that the orthogonal projections of any point in 
  $\mathscr{D}_{\Theta_1} \cap \mathscr{D}_{\Theta_2}$ onto two 
  different codimension 1 subspaces, $\Theta_1^\perp$ and 
  $\Theta_2^\perp$, are bounded. Therefore, a closed 
  neighborhood of finite radius about the intersection is compact too. 
  Hence $\overline{B}(\mathscr{D}_{\Theta_1},1) \cap 
  \overline{B}(\mathscr{D}_{\Theta_2},1)$ is compact. We will use this 
  below.

  Assume, contrary to the claim, that for any $r_0>0$ the intersection $\cap_k 
  \overline{B}(\mathscr{D}_{\Theta_k},2r_0)$ is  non-empty. Then 
  there is a sequence $\xi^1, \xi^2, \ldots \in \R^n$ such
  that $\sup_{k}d(\xi^\ell, \mathscr{D}_{\Theta_k})$
  approaches zero.
By the 
  compactness of $\overline{B}(\mathscr{D}_{\Theta_1},1) \cap 
  \overline{B}(\mathscr{D}_{\Theta_2},1)$ we may assume that $\xi^\ell$ 
  converges to some $\xi$. Then $\xi \in \mathscr{D}_{\Theta_k}$ for all 
  $k$ since the latter are closed sets. This contradicts the assumption 
  that the  intersection of the
  \(\mathscr{D}_{\Theta_k}\) is empty and establishes \eqref{eq:89}.\\
  
  Let $\psi_k : \Theta_k^\perp \to {[{0,1}]}$ be smooth and such that 
  $\psi_k(\xi_{\Theta_k^\perp}) = 0$ if $d(\xi_{\Theta_k^\perp}, 
  \mathscr{D}_{\Theta_k}) \le r_0$ and $\psi_k(\xi_{\Theta_k^\perp}) 
  = 1$ if $d(\xi_{\Theta_k^\perp}, \mathscr{D}_{\Theta_k}) \geqslant 2r_0$. 
  Set
\begin{equation}\label{partitionOfUnityDef}
  \Psi_1(\xi) = \psi_1(\xi_{\Theta_1^\perp}), \qquad \Psi_{k+1}(\xi) = 
  \psi_{k+1}(\xi_{\Theta_{k+1}^\perp}) \prod_{\ell=1}^k \big( 1 - 
  \psi_\ell(\xi_{\Theta_\ell^\perp}) \big)
\end{equation}
  where $\xi_{\Theta_\ell^\perp} = \xi - 
  (\xi\cdot\Theta_\ell)\Theta_\ell \in \Theta_\ell^\perp$. Then 
  $\Psi_k:\R^n\to{[{0,1}]}$ smoothly and $\Psi_k \equiv 0$ on 
  $\overline{B}(\mathscr{D}_{\Theta_k},r_0)$.

  Note that $1-\psi_k \in C^\infty_0(\Theta_k^\perp)$ for $k=1,2$ and 
  $\xi\mapsto\psi_k(\xi_{\Theta_k^\perp})$ is constant in the direction of $\Theta_k$. Thus, given any 
  $\Theta\in\mathbb{S}^{n-1}$, Corollary \ref{constantInDirectionNorm} 
  implies that
\begin{equation}\label{nonParallelMult}
  \norm{ \f{\Theta^\perp}{M_{1-\psi_k}f} }_{\Theta(1,2)} \le 
  \norm{\finv{\Theta_{k\perp}}{\{1-\psi_k\}}}_{\Theta_{k\perp}(1,\infty)} 
  \norm{ \f{\Theta^\perp}{f} }_{\Theta(1,2)}
\end{equation}
  for some direction $\Theta_{k\perp}$ in the $(\Theta, \Theta_k)$-plane 
  perpendicular to $\Theta_k$ when $\Theta\not\parallel\Theta_k$, and
\begin{equation}\label{parallelMult}
  \norm{ \f{\Theta^\perp}{M_{1-\psi_k}f} }_{\Theta(1,2)} \le 
  \sup_{\Theta_k^\perp} \abs{1-\psi_k} \norm{ \f{\Theta^\perp}{f} 
  }_{\Theta(1,2)}
\end{equation}
  when $\Theta\parallel\Theta_k$. Recall that in the first case the 
  $\Theta_{k\perp}(1,\infty)$-norm is taken in the $n-1$ dimensional 
  space $\Theta_k^\perp$. In both cases the multiplier norm, which we 
  denote by $C_k = C_k(\Theta,\Theta_k)$, is finite since $1-\psi_k$ is 
  smooth and compactly supported in $\Theta_k^\perp$, so in particular 
  $\finv{\Theta_{k\perp}}{\{1-\psi_k\}}$ is a Schwartz test function.
  
  Thus, by \eqref{partitionOfUnityDef}, \eqref{nonParallelMult} and 
  \eqref{parallelMult}
\begin{align*}
  &\norm{ \f{\Theta^\perp}{M_{\Psi_1}f} }_{\Theta(1,2)} \le \norm{ 
  \f{\Theta^\perp}{f + M_{1-\psi_1}f} }_{\Theta(1,2)} \le (1+C_1) 
  \norm{ \f{\Theta^\perp}{f} }_{\Theta(1,2)}, \\ &\norm{ 
  \f{\Theta^\perp}{M_{\Psi_2}f} }_{\Theta(1,2)} \le C_1 \norm{ 
  \f{\Theta^\perp}{M_{\psi_2}f} }_{\Theta(1,2)} \le C_1(1+C_2) \norm{ 
  \f{\Theta^\perp}{f} }_{\Theta(1,2)}.
\end{align*}
  We cannot apply the same argument to  $M_{\Psi_3}, 
  M_{\Psi_4}, \ldots$ because the multipliers $1-\psi_k$ are not 
  necessarily compactly supported in $\Theta_k^\perp$. Instead we note 
  that $K = \supp (1-\psi_1)(1-\psi_2) \subset \R^n$ is compact. So 
  $\Psi_{k+1} \in C^\infty_0(B(K,1))$ for $k\geqslant 2$. Lemma 
  \ref{multiplierNorm} then implies that
\[
  \norm{ \f{\Theta^\perp}{M_{\Psi_{k+1}}f} }_{\Theta(1,2)} \le 
  \frac{1}{\sqrt{2\pi}} \norm{ \finv{\Theta}{\Psi_{k+1}} 
  }_{\Theta(1,\infty)} \norm{ \f{\Theta^\perp}{f} }_{\Theta(1,2)}
\]
  where the first norm is finite since $\finv{\Theta}{\Psi_{k+1}} \in 
  \mathscr{S}(\R^n)$. So the multipliers are bounded in all directions: 
  there are finite $C'_k =C'_k(\Theta,\Theta_1,\ldots,\Theta_k)$ such 
  that $\norm{ \f{\Theta^\perp}{M_{\Psi_k}f}}_{\Theta(1,2)} \le C'_k 
  \norm{ \f{\Theta^\perp}{f}}_{\Theta(1,2)}$ for all $k$ and any 
  $\Theta\in\mathbb{S}^{n-1}$.
  
  For the last claim sum the $\Psi_k$ all up to get
\[
  \sum_k \Psi_k(\xi) = 1 - \prod_k 
  \big(1-\psi_k(\xi_{\Theta_k^\perp})\big).
\]
  Since $\supp (1-\psi_k) \subset \overline{B}(\mathscr{D}_{\Theta_k}, 
  2r_0)$ and the intersection of the latter is empty, the product 
  vanishes everywhere.
\end{proof}

  \bigskip We now have all  the necessary ingredients for
  the proof of the main theorem of this section.

\begin{proof}[Proof of Theorem \ref{th:other}]
  Let $P$ be admissible of degree $N\geqslant 1$ and $P_N$ its principal term. 
  Then propositions \ref{algProp} and \ref{epsNeighborhood} imply the 
  existence of a finite set of directions 
  $\Theta_k\in\mathbb{S}^{n-1}(\R^n) \setminus P_N^{-1}(0)$, 
  $k=1,\ldots,m$, an associated partition of unity $\Psi_k$ and a 
  constant $r_0>0$.
  
  Set $f_k = M_{\Psi_k}f$. Then $f = \sum_k f_k$, $\f{}{f_k}(\xi) = 0$ 
  when $d(\xi,\mathscr{D}_{\Theta_k}) \le r_0$, and
\[
  \norm{ \f{\Theta^\perp}{f_k} }_{\Theta(1,2)} \le C_{k,\Theta} \norm{ 
  \f{\Theta^\perp}{f} }_{\Theta(1,2)} \le C_{k,\Theta} \sqrt{d_s} 
  \norm{f}{L^2(D_s)}
\]
  for any $\Theta\in\mathbb{S}^{n-1}$ by Proposition 
  \ref{epsNeighborhood} and Lemma \ref{mixedToUniformNorms}.

  By Condition \ref{twoCylindersCompatibility} of the admissibility 
  definition in \ref{admissibleSymbol} there is $\varepsilon>0$ such 
  that $\f{}{f_k} = 0$ on $\mathscr{B}_{\Theta_k, \varepsilon}$. Let 
  $u_k$ be the solution to $P(D)u_k = f_k$ given by Proposition 
  \ref{solExistence}. We have
\[
  \norm{ \f{\Theta_k^\perp}{u_k} }_{\Theta_k(\infty,2)} \le 
  \frac{N}{\varepsilon} \norm{ \f{\Theta_k^\perp}{f_k} 
  }_{\Theta_k(1,2)}
\]
  by that same proposition. The theorem follows by setting $u = \sum_k 
  u_k$ since
\[
  \norm{u_k}_{L^2(D_r)} \le \sqrt{d_r} \norm{ \f{\Theta_k^\perp}{u_k} 
  }_{\Theta_k(\infty,2)}
\]
  by Lemma \ref{mixedToUniformNorms}.
\end{proof}

\begin{remark}
  The same proof gives $\norm{u}_{L^q(D_r)} \le C d_r^{1/q} d_s^{1/p} 
  \norm{\f{}{f}}_{L^p(\R^n)}$ if $p\le2\le q$ and $p^{-1}+q^{-1}=1$.
\end{remark}

\section{Examples}\label{sec:examples}

  We describe estimates for a few specific PDE's below. Some of the 
  estimates follow directly from Theorem \ref{th:main} or 
  Theorem \ref{th:other}. Others illustrate how the
  method can be applied in different settings.

\begin{example}
  The inhomogeneous Helmholtz equation $(\Delta+k^2)u = f$
  is the motivating example for this work. The equation is
  rotation and translation invariant, and scales simply
  under dilations. Estimates in weighted norms typically
  share none of these properties\footnote{Homogeneous
    weights, e.g. \(||\;|x|^{\delta}f||_{L^{2}}\), retains
    scaling properties at the cost of allowing
    singularities at the origin. They are invariant under
    rotations about the origin, but not about any other
    point}. For this reason, the dependence of the
  estimate on wavenumber \(k\), which is the physically
  relevant parameter, is not clear. However, an estimate
  that comes from Theorem \ref{th:main} or Theorem
  \ref{th:other}, with \(k=1\), i.e.
\[
  \norm{u}_{L^2(D_r)} \le C_{1}\sqrt{d_r d_s} \norm{f}_{L^2(D_s)}
\]
for  \(f\) with $\supp f \subset D_s \subset \R^n$, immediately implies 
\[
  \norm{u}_{L^2(D_r)} \le C \frac{\sqrt{d_r d_s}}{k} 
  \norm{f}_{L^2(D_s)}
\]
by simply noting that  $U(x) = u(kx)$ satisfies 
\begin{eqnarray*}
  (\Delta+k^2)U = k^{2}f(kx)
\end{eqnarray*}
and using the fact that  the diameters scale as distance (i.e.
\(d\mapsto kd\)) and  \(L^{2}\) norms like distance to
the power  \(\frac{n}{2}\).

A second advantage is that \textit{diameter} in Theorems
\ref{th:main} and \ref{th:other} means the length of the
intersection of any line with \(D_{r}\) or \(D_{s}\). This
is particularly appropriate for a source that is supported
on a union of small sets that are far
apart\footnote{Locating well-separated sources and
  scatterers is one of the most well-studied applied
  inverse problems modelled by the Helmholtz
  equation\cite{MUSIC}.}. In weighted norms, the parts of
the source that are far from the \textit{origin} at which
the weights are based, will have large norm because of
their location, yet their contribution to the solution
\(u\) or its far field (asymptotics used in scattering
theory and inverse problems) is no larger than it would be
if it were located at the origin. Insisting that our
estimates share all the invariance properties of the
underlying PDE eliminates these artificial differences
between the physics and the mathematics\footnote{Honesty demands
  that we acknowledge that our domain dependent estimates
  provide semi-norms, rather than norms, so we are not
  ready to give up weighted norms and Besov type norms
  entirely.}.

Estimates of the $L^q$ norms of \(u\) in  in terms of $L^p$ norms of
\(\f{}f\) are sometimes useful as well \cite{corners}. For
$p^{-1}+q^{-1}=1$, $p\le 2 \le q$, our methods give
\[
  \norm{u}_{L^q(D_r)} \le C k^{-1} d_r^{1/q} d_s^{1/p} 
  \norm{\f{}{f}}_{L^p(\R^n)}.
\]
\end{example}

\begin{example}
 The Bilaplacian is a fourth order PDE that arises in the
 theory of elasticity and in the modelling of fluid flow
 (Stokes flow). We include a spectral parameter
 \(\lambda\) and an 
  external force $f$:
\[
  (\Delta^2 - \lambda^2)u = f.
\]
Let us show that the admissibility 
  conditions for Theorem \ref{th:other} given by Definition 
  \ref{admissibleSymbol} are satisfied.

  Assume $\lambda > 0$ and write $\Delta^2-\lambda^2 = P(D)$, and so 
  $P(\xi) = \abs{\xi}^4 - \lambda^2$. 
  Let $\Theta\in\mathbb{S}^{n-1}$ and for $\xi_{\Theta^\perp} \in 
  \Theta^\perp$ write
\[
  p(\tau) = P(\tau\Theta+\xi_{\Theta^\perp}) = (\tau^2 + 
  \abs{\xi_{\Theta^\perp}}^2 - \lambda) (\tau^2 + 
  \abs{\xi_{\Theta^\perp}}^2 + \lambda).
\]
  The roots $\tau=\tau_j$ are easily seen to be
\begin{align*}
  &\tau_1 = \sqrt{\lambda - \abs{\xi_{\Theta^\perp}}^2}, &&\tau_3 = 
  i\sqrt{\lambda + \abs{\xi_{\Theta^\perp}}^2}, \\ &\tau_2 = 
  -\sqrt{\lambda - \abs{\xi_{\Theta^\perp}}^2}, &&\tau_4 = 
  -i\sqrt{\lambda + \abs{\xi_{\Theta^\perp}}^2},
\end{align*}
  where the square root has been chosen to return a non-negative real 
  part and mapping the negative real axis to the imaginary axis in the 
  upper half-plane.
  
  The derivative in the direction $\Theta$ is given by $p'(\tau) = 
  4(\tau^2 + \abs{\xi_{\Theta^\perp}}^2)\tau$. Hence
\[
  \abs{p'(\tau_j)} = 4 \lambda \sqrt{\lambda - 
  \abs{\xi_{\Theta^\perp}}^2}
\]
  for $j=1,2$, and
\[
  \abs{p'(\tau_j)} = 4 \lambda \sqrt{\lambda + 
  \abs{\xi_{\Theta^\perp}}^2}
\]
  for $j=3,4$. Note that in the latter case $\abs{p'(\tau_j)} \geqslant 4 
  \lambda^{3/2}$ for all  \(\xi_{\Theta^\perp}\) .

  If $\varepsilon < 4\lambda^{3/2}$ we see that
\begin{align}\nonumber
  \mathscr{B}_{\Theta,\varepsilon} &= \left\{ \xi\in\R^n \,\middle|\, 
  \lambda - \frac{\varepsilon^2}{16\lambda^2} \le 
  \abs{\xi_{\Theta^\perp}}^2 \le \lambda + 
  \frac{\varepsilon^2}{16\lambda^2} \right\}, \\\label{eq:90}
  \mathscr{D}_\Theta &= 
  \left\{ \xi\in\R^n \,\middle|\, \abs{\xi_{\Theta^\perp}}^2 = \lambda 
  \right\}.
\end{align}
  For any $r_0 > 0$, set $\varepsilon < 4 \lambda^{5/4} r_0^{1/2}$. For any $\xi 
  \in \mathscr{B}_{\Theta,\varepsilon}$ set $\zeta = 
  (\xi\cdot\Theta)\Theta + \lambda^{1/2} \xi_{\Theta^\perp} / 
  \abs{\xi_{\Theta^\perp}}$. Then $\zeta \in \mathscr{D}_\Theta$ and
\[
  \abs{\xi-\zeta} = \abs{ \abs{\xi_{\Theta^\perp}} - \lambda^{1/2} } = 
  \frac{\abs{ \abs{\xi_{\Theta^\perp}}^2 - \lambda }}{ 
  \abs{\xi_{\Theta^\perp}} + \lambda^{1/2} } \le 
  \frac{\varepsilon^2/(16\lambda^2)}{\lambda^{1/2}} < r_0
\]
  and so $\mathscr{B}_{\Theta,\varepsilon} \subset B(\mathscr{D}_\Theta, 
  r_0)$ whenever $\varepsilon < \min( 4\lambda^{3/2}, 4 \lambda^{5/4} 
  r_0^{1/2} )$. Condition \ref{twoCylindersCompatibility} in Definition 
  \ref{admissibleSymbol} is thus satisfied. Condition
  \ref{compactDirections} is an easy consequence of
  \eqref{eq:90}. Combining the
  estimate from Theorem \ref{th:other} with  a 
  scaling argument similar to the previous example yields
  \begin{eqnarray*}
     \norm{u}_{L^2(D_r)} \le C \frac{\sqrt{d_r d_s}}{\lambda^{\frac{3}{2}}} 
  \norm{f}_{L^2(D_s)}.
  \end{eqnarray*}
\end{example}

\begin{example}
  The operator $P(D) = D_1^2D_2^2 - 1$ is not \emph{simply
    characteristic}, and its zeros are not \emph{uniformly
    simple}, as defined in definitions 4.2 and 6.2 by
  Agmon and H\"ormander \cite{Agmon-Hormander} or Section
  14.3.1 in H\"ormander's book \cite{Hormander2}. This is
  because the characteristic variety $P^{-1}(0)$ has two
  different branches approaching a common asymptote
  (Figure \ref{exNonSimply}). Thus the Besov style
  estimates established using uniform simplicity do not
  apply to this operator. We show below that the conditions
  in Definition \ref{admissibleSymbol} are satisfied, so that the
  estimate of Theorem \ref{th:other} holds. As we remarked
  in the introduction, the Besov style estimates of
  \cite{Agmon-Hormander} are a specialization of
  \eqref{eq:91}, and therefore a consequence of Theorem
  \ref{th:other}.\\

\begin{figure}
\begin{center}
  \includegraphics{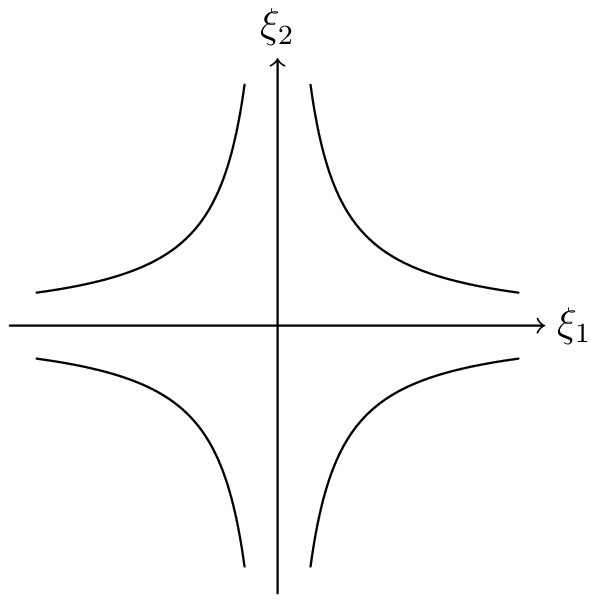}
\end{center}
\caption{Characteristic variety of $D_1^2D_2^2 - 1$.}
\label{exNonSimply}
\end{figure}

It is straightforward to check that \(P(\xi)\) is
nonsingular. We will verify the conditions in Definition
\ref{admissibleSymbol} for $\Theta\in\mathbb{S}^{n-1}$,
with $\Theta_1\neq 0$ and $\Theta_2 \neq 0$, and calculate
$\abs{\Theta\cdot\nabla P(\xi)}$ for every  \(\xi\) in
$P(\xi)=0$ with a fixed $\xi_{\Theta^\perp}$ component. A
glance at Figure \ref{exNonSimply} shows that there will
between two and four real \(\xi\)'s satisfying $P(\xi)=0$ that have
the same  $\xi_{\Theta^\perp}$ component.  We begin by
parameterizing the complex characteristic variety
\[
  P^{-1}(0)=\{ (s, s^{-1}), (s, -s^{-1}) \in \C^2 \mid s\in\C, s\neq 0\}.
\]
Next we project each point in the variety,
$\xi = (s,\pm s^{-1})$, onto
$\Theta^\perp = \{ b(-\Theta_2, \Theta_1) \mid
b\in\R\}$. Its  \(\Theta^\perp\) component is\footnote{Not 
all complex roots will project to
  $\Theta^\perp$ embedded in the reals. But we are only
  interested in the part of the characteristic variety
  that does.}

\begin{eqnarray}\label{eq:93}
  \xi_{\Theta^\perp} = \big( -\Theta_2 s \pm \Theta_1 s^{-1} \big) 
  (-\Theta_2, \Theta_1) =: b(-\Theta_2,
  \Theta_1).
  \end{eqnarray}
  To verify conditions about
  $\mathscr{B}_{\Theta,\varepsilon}$, we want to
  parameterize the points on the
  variety  \(P^{-1}(0)\) in terms of their
  \(\xi_{\Theta^\perp}\) component, which is parameterized
  by  \(b\). So we use \eqref{eq:93} to solve for 
     $s = s(b)$.  The four (complex) roots $\xi = (s(b), \pm 
  s(b)^{-1})$ of $P$ on the line defined by $\xi_{\Theta^\perp} = 
  b(-\Theta_2,\Theta_1)$ are
\begin{align*}
  &\xi^{(1)} = \left( \frac{-b+\sqrt{b^2+4\Theta_1\Theta_2}}{2\Theta_2}, 
  \frac{b+\sqrt{b^2+4\Theta_1\Theta_2}}{2\Theta_1} \right), \\ 
  &\xi^{(2)} = \left( \frac{-b-\sqrt{b^2+4\Theta_1\Theta_2}}{2\Theta_2}, 
  \frac{b-\sqrt{b^2-4\Theta_1\Theta_2}}{2\Theta_1} \right), \\ 
  &\xi^{(3)} = \left( \frac{-b+\sqrt{b^2-4\Theta_1\Theta_2}}{2\Theta_2}, 
  \frac{b+\sqrt{b^2-4\Theta_1\Theta_2}}{2\Theta_1} \right), \\ 
  &\xi^{(4)} = \left( \frac{-b-\sqrt{b^2-4\Theta_1\Theta_2}}{2\Theta_2}, 
  \frac{b-\sqrt{b^2-4\Theta_1\Theta_2}}{2\Theta_1} \right).
\end{align*}
  
  The derivative in the direction $\Theta$ at any root $\xi$ having 
  $\xi_1\xi_2=\pm1$ is
\[
  \Theta\cdot\nabla P(\xi) = \pm 2 (\Theta_2 \xi_1 + \Theta_1 \xi_2).
\]
  Hence, after simplification,
\begin{align*}
  &\Theta\cdot\nabla P(\xi^{(1)}) = 2\sqrt{b^2+4\Theta_1\Theta_2},\\
  &\Theta\cdot\nabla P(\xi^{(2)}) = -2\sqrt{b^2+4\Theta_1\Theta_2},\\
  &\Theta\cdot\nabla P(\xi^{(3)}) = -2\sqrt{b^2-4\Theta_1\Theta_2},\\
  &\Theta\cdot\nabla P(\xi^{(4)}) = 2\sqrt{b^2-4\Theta_1\Theta_2}.
\end{align*}
  So $\abs{\Theta\cdot\nabla P(\xi)} \geqslant 2 \abs{b^2 - 
  4\abs{\Theta_1\Theta_2}}^{1/2}$ at any root $\xi$ with 
  $\xi_{\Theta^\perp} = b(-\Theta_2,\Theta_1)$.

 Now we have explicit descriptions of the sets that appear
 in Definition \ref{admissibleSymbol} and can verify the
 hypotheses of Theorem \ref{th:other}; namely,
\begin{align*}
  &\mathscr{D}_\Theta = \{ \xi\in\R^2 \mid \xi_{\Theta^\perp} = 
  b(-\Theta_2,\Theta_1), \quad b^2 - 4\abs{\Theta_1\Theta_2} = 0 \},\\
  &\mathscr{B}_{\Theta,\varepsilon} = \{ \xi\in\R^2 \mid 
  \xi_{\Theta^\perp} = b(-\Theta_2,\Theta_1), \quad \abs{b^2 - 
  4\abs{\Theta_1\Theta_2}} \le \varepsilon^2/4 \}
\end{align*}
as long as we choose
$\varepsilon^2 < 16 \abs{\Theta_1\Theta_2}$. For any
$r_0>0$ let
$\varepsilon^2 \le
8r_0\sqrt{\abs{\Theta_1\Theta_2}}$. Then if
$\xi \in \mathscr{B}_{\Theta,\varepsilon}$, we have (with
$\xi_{\Theta^\perp} = b(-\Theta_2,\Theta_1)$)
\[
  d(\xi,\mathscr{D}_\Theta) \le \abs{b - 
  2\sqrt{\abs{\Theta_1\Theta_2}}} = \frac{\abs{b^2 - 
  4\abs{\Theta_1\Theta_2}}}{b + 2\sqrt{\abs{\Theta_1\Theta_2}}} \le 
  \frac{\varepsilon^2}{8\sqrt{\abs{\Theta_1\Theta_2}}} \le r_0
\]
  for $b \geqslant 0$, and similarly $d(\xi,\mathscr{D}_\Theta) \le \abs{b + 
  2\sqrt{\abs{\Theta_1\Theta_2}}} \le r_0$ for $b\le 0$. Hence 
  $\mathscr{B}_{\Theta,\varepsilon} \subset 
  \overline{B}(\mathscr{D}_\Theta,r_0)$ for any $r_0>0$ if 
  $\varepsilon^2 \le 8r_0\sqrt{\abs{\Theta_1\Theta_2}}$ and 
  $\varepsilon^2 < 16\abs{\Theta_1\Theta_2}$, so we have
  verified Condition \ref{twoCylindersCompatibility}, and
  Condition  \ref{compactDirections} is automatic in two
  dimensions, so we are finished.
\end{example}
\goodbreak

\begin{example}
  The Faddeev operator is ubiquitous in the area of inverse problems. 
  Its solution enables the construction of the so-called \emph{Complex 
  Geometric Optics} solutions to the Laplace equation that are used to 
  prove uniqueness for many inverse scattering and inverse boundary 
  value problems. See \cite{Faddeev65} for an early application to 
  scattering theory, Sylvester and Uhlmann \cite{Sylvester-Uhlmann} and 
  Nachmann \cite{Nachmann88} for its application to solving the 
  \emph{Calder\'on problem} \cite{Calderon}, and \cite{Uhlmann2009} for 
  a review of more recent developments in that area.
  
  The simplest form, as introduced by Calder\'on is
  \begin{eqnarray}\label{eq:94}
    \left(\Delta + 2\zeta\cdot\nabla\right)u=f
  \end{eqnarray}
with  \(\zeta\in\C^{n}\)  satisfying
\(\zeta\cdot\zeta=0\). It has complex coefficients, but
 setting  \(v=e^{i\Im{\zeta}\cdot x}u\) and
 \(g=e^{i\Im{\zeta}\cdot x}f\) results in
 \begin{eqnarray*}
   \left(\Delta + 2\Re{\zeta}\cdot\nabla\right)v=g
 \end{eqnarray*}
 which has real coefficients. Moreover, \(u\) and \(v\)
 have the same \(L^{p}\) norms, as do \(f\) and \(g\).
 The symbol and its gradient are
\begin{eqnarray*}
  P(\xi) &=& -\xi\cdot(\xi-2i\Re{\zeta})
\\
  \nabla P &=& 2(-\xi+i\Re{\zeta})
\end{eqnarray*}
so  \(\nabla P\) has no real zeros. Thus \(P\) has no
real double characteristics and Theorem \ref{th:main}
applies. Because the equation, and the estimates, dilate
simply, scaling again gives the exact dependence on  \(\zeta\). 
\begin{eqnarray}\label{eq:92}
  ||v||_{L^{2}(D_{r})}\le C
  \frac{\sqrt{d_{r}d_{s}}}{|\Re{\zeta|}}||g||_{L^{2}(D_{s})} 
\end{eqnarray}
with $\supp f \subset D_s$. Here $d_r,
d_s$ are the diameters of the open sets $D_r,
D_s$.  We may, of course, replace \(v\) by \(u\) and \(g\)
by \(f\).\\

In some applications, the condition \(\zeta\cdot\zeta=0\)
is replaced by \(\zeta\cdot\zeta=\lambda\). As the
gradient of  \(P\) is still nowhere vanishing, Theorem
\ref{th:main} still applies, and the estimates still
scale, but it is not clear how the estimates depend on the
ratio  \(\frac{\Re{\zeta}}{\sqrt{\lambda}}\). A direct
calculation shows that \eqref{eq:92} still holds. In addition,
Remark \ref{mixedToUniformNormsLp} also applies here, so
we have for $\frac{1}{p}+\frac{1}{q}=1, p \le 2 \le q$,
\[
  \norm{u}_{L^q(D_r)} \le \frac{ d_r^{1/q} d_s^{1/p} }{\abs{\Re{\zeta}}} 
  \norm{ \f{}{f} }_{L^p(\R^n)}.
\]

  Equation  \eqref{eq:94} has a special direction. We
  expect a solution to decay exponentially in 
  the direction  $\Theta = \Re \zeta / \abs{\Re \zeta}$,
  so an anisotropic estimate is natural here. 
  Taking the Fourier transform in the  \(\Theta^{\perp}\) hyperplane 
  reduces \eqref{eq:94} to an ordinary 
  differential equation which can be factored into the product of two 
  first order operators. Then using \eqref{eq:45} for one of the factors 
  and \eqref{eq:46} for the other gives the estimate
\begin{equation}\label{eq:anisoFaddeev}
  \norm{\f{\Theta^\perp}{u}}_{\Theta(\infty,2)} \le 
  \frac{\norm{\f{\Theta^\perp}{f}}_{\Theta(1,2)}} 
  {\inf_{\xi_{\Theta^\perp} \in \Theta^\perp} \abs{\abs{\Re\zeta} + 
  \sqrt{\abs{\Im\zeta+\xi_{\Theta^\perp}}^2 - \lambda}}} \le 
  \frac{\norm{\f{\Theta^\perp}{f}}_{\Theta(1,2)}}{\abs{\Re\zeta}}
\end{equation}
  when $\lambda\in\R$. This estimate implies \eqref{eq:92} by 
  \eqref{eq:71}.

\end{example}

\bigskip Theorems \ref{th:main} and \ref{th:other} apply to 
  scalar valued PDE's only, but the method can be applied
  to systems. The next proposition could be substantially
  more general, but it is enough to establish estimates for
  the Dirac system.
  
  \begin{proposition}\label{th:dirac}
  Consider a constant coefficient first order system
  \begin{eqnarray}
    \nn
    \bm A(D) =
    \sum_{j=1}^{n}\bm A_{j}\frac{\partial}{\partial x_{j}} + \bm B
  \end{eqnarray}
with $\bm A_1,\ldots,\bm A_n, \bm B \in \C^{n\times n}$ and suppose that, for some  \(k\),
\begin{eqnarray}
  \nn
  \bm M(\xi) = \bm A_{k}^{-1}\left(\sum_{j\ne k}\bm A_{j}\xi_{j}+\bm B\right)
\end{eqnarray}
is normal for all  \(\xi\)\footnote{Equivalently, for some
   \(k\) and all  \(j\),  \(\bm A_{k}^{-1}\bm A_{j}\) and
   \(\bm A_{k}^{-1}\bm B\) are normal} . Then, there is a constant
\(C\), such that for every  \(f\), there exists  \(u\)
solving
\begin{eqnarray}
  \nn
  \bm A(D)u &=&f
\\\na{and}\label{eq:75}
||\ftp{u}||_{\Theta(\infty,2)}&\le&C||\ftp{f}||_{\Theta(1,2)}
\end{eqnarray}
where  \(\Theta\) is the unit vector in the  \(k\)th
coordinate direction, and consequently, for  \(f\)
supported in  \(D_{r}\) and any  \(D_{s}\), 
\begin{equation}
  ||u||_{L^2(D_{s})}\le C\sqrt{d_{r}d_{s}} ||f||_{_{L^2(D_{r})}}
\end{equation}
  where \(d_{i}\)
  is the diameter of \(D_{i}\)
  and \(C\)
  is a constant that depends only \(\bm A(D)\)
  and the dimension \(n\).
\end{proposition}
\begin{proof}
 We take the partial Fourier transform in the  \(\Theta^\perp\)
hyperplane, and note that the vector \(\tilde{u}:=\ftp{u}\) must satisfy
\begin{eqnarray}
  \label{eq:25}
    \frac{\partial}{\partial x_{k}}\tilde{u} + \bm M(\xi)\tilde{u} &=&  \bm A_{k}^{-1}\tilde{f}
\end{eqnarray}
and simply write the solution
\begin{eqnarray}\label{eq:74}
  \tilde{u}(t,\xi) &=&
                     \int_{-\infty}^{t}e^{\bm M(\xi)(t-s)}P^{+} \bm A_{k}^{-1}\tilde{f}(s,\xi)ds
\\
       &-& \int_{t}^{\infty}e^{\bm M(\xi)(t-s)}P^{-} \bm A_{k}^{-1}\tilde{f}(s,\xi)ds
\end{eqnarray}
where \(P^{+}(\xi)\)
is the orthogonal projection onto the
\(\Re{\lambda}\ge 0\)
eigenspace of  \(\bm M(\xi)\) and \(P^{-}(\xi)\)
is the orthogonal projection onto the
\(\Re{\lambda}< 0\)
eigenspace. The projections need not be continuous
functions of  \(\xi\), but they need only be measurable
for the formula to make sense. The fact that  \(\bm M(\xi)\)
is normal guarantees that the sum of the projections is
the identity, and therefore  that \(\tilde{u}\) really does
solve (\ref{eq:25}). The estimate (\ref{eq:75}) follows
immediately from the formula (\ref{eq:74}) and the fact
that the orthogonal projections have norm one or zero.
\end{proof}

\begin{example}\label{exDirac}
The 4x4 Dirac operator may be written as a prolongation of
the curl operator
\begin{eqnarray}
  \bm D =
  \begin{pmatrix}
    \nabla\times&-\nabla\\\nabla\cdot&0
  \end{pmatrix}.
  \end{eqnarray}
Alternatively, we may express the first order system as
\begin{eqnarray}
  \nn
  \bm D-i\omega \bm I = \sum_{j=1}^{3}\bm A_{j}\frac{\partial}{\partial x_{j}} -
  i\omega \bm I
\end{eqnarray}
where $\bm I$ is the $4\times 4$ identity matrix and
\begin{eqnarray*}
  \bm P =
        \begin{pmatrix}
          0&-1\\1&0
        \end{pmatrix},
\quad
\bm A_{1} =
        \begin{pmatrix}
          0 &\bm P\\ \bm P & 0
        \end{pmatrix},
\quad
\bm A_{2} =
        \begin{pmatrix}
          0&-\bm I \\ \bm I &0
        \end{pmatrix},
\quad
\bm A_{3} =
        \begin{pmatrix}
          \bm P & 0 \\ 0 & -\bm P
        \end{pmatrix}.
\end{eqnarray*}

It is straightforward to verify that each \(\bm A_{j}\) is
skew, \(\bm A_{j}^{2}=-\bm I\) and that \(\bm A_{i} \bm A_{j}=\pm \bm A_{k}\)
when all three indices are different. These facts
guarantee the hypotheses of Proposition \ref{th:dirac},
and hence the estimate (\ref{eq:75}) for a solution \(u\)
of
\begin{eqnarray}
  \nn
  \left(\bm D-i\omega \bm I\right)u = f.
\end{eqnarray}
\end{example}
\begin{example}[Non-Example]
  We show that the estimates \eqref{eq:53} do not hold for
  the Laplacian in 3 dimensions, which has a double
  characteristic. Suppose that \(f\) is compactly
  supported and
  \begin{eqnarray*}
    \Delta u = f
  \end{eqnarray*}
In 3 dimensions, 
\begin{eqnarray*}
  u(x) = \int\frac{f(y)}{|x-y|}dy + H(x)
\end{eqnarray*}
where \(H\) is a harmonic polynomial. For compactly
supported \(f\), the estimates \eqref{eq:92} would imply
that \(u\) grows no faster than \(1/|x|\), so \(H\) must
be zero. We choose \(f\) to be identically one on the ball
of radius \(A\) centered at the origin. In this case,
\(||f||_{L^{2}(A)}\) is
\(\sqrt{\frac{4}{3}\pi{}A^{3}}\). We next compute
\(||u||_{L^{2}(B_{R}(c))}\), with \(R>>A\) and
\(|c| = 2|R|\). For \(x\in B_{R}(c)\)
  
\begin{eqnarray*}
  |u(x)| &>& \frac{1}{2}\frac{\int f(y)}{R}
\\\na{so that}
 ||u||_{L^{2}(B_{R}(c))}&\ge&\frac{1}{2}
  \frac{(\frac{4}{3}\pi{}A^{3})(\sqrt{\frac{4}{3}\pi{}R^{3}})}{R}
=\frac{1}{2} (\frac{4}{3}\pi)^{\frac{3}{2}} A^{3}R^{\frac{1}{2}}
\end{eqnarray*}
Estimate \eqref{eq:53} would imply that 
\begin{eqnarray*}
\frac{1}{2}(\frac{4}{3}\pi)^{\frac{3}{2}} A^{3}R^{\frac{1}{2}}
\le C\sqrt{AR}A^{\frac{3}{2}} = CA^{\frac{5}{2}}R^{\frac{1}{2}}
\end{eqnarray*}
which is impossible for large  \(A\). 
\end{example}

\section{Conclusions}

We have introduced a technique for proving some simple,
translation invariant estimates, which scale naturally,
and can therefore be directly interpreted for physical
systems and remain meaningful in any choice of units.
Such estimates are necessary because because physical
principles dictate that the fields should store
finite energy in a bounded region (i.e. solutions should
be locally \(L^{2}\)) and radiate finite power, which
implies that they should decay at least as fast as
\(r^{-\frac{n-1}{2}}\) near infinity. We have
replaced weighted norms by estimates on bounded regions
which depend on the diameter of these regions. Because
the estimates  depend on natural geometric quantities,
which rotate, dilate and translate in natural ways, the
estimates themselves have the same symmetries as the
underlying PDE models. The  \(L^{2}\) estimates are based on anisotropic
estimates that are analogous to those  that hold for a
parameterized ODE, so it is reasonable to expect them to
hold for all simply characteristic PDE, but we have not
proven any theorems in that generality, nor produced
examples to show that more restrictions are
necessary. Indeed, we expect that these estimates are true
for many more PDE's and systems than we have covered
here.\\

Theorem \ref{th:main} can certainly be extended to allow
 first order terms with complex coefficients using the
 change of dependent variable in the line following
 \eqref{eq:94}, but we do not know if we can allow other
 complex coefficients as well.\\

 Theorem \ref{th:other} includes many technical
 assumptions that we doubt are necessary. The hypothesis
 that the characteristic variety is non-singular over
 \(\C^{n}\) rather than \(\R^{n}\) is clearly not
 necessary, but we don't know of a simple replacement.
 The admissibility conditions in Definition
 \ref{admissibleSymbol} were chosen to facilitate the
 proof, and enforce a certain uniform behavior outside
 compact sets, somewhat similar to Agmon-H\"ormander's
 \textit{uniformly simple} hypothesis.  In two dimensions,
 where Condition \ref{compactDirections} is automatically
 satisfied, we are not aware of any nonsingular polynomial
 $P:\C^2\to\C$ for which Condition
 \ref{twoCylindersCompatibility} does not hold.
  
 We have only given one example of a system of PDE's. The
 estimates for the Dirac system were particularly easy
 because, for any direction \(\Theta\),
 the resulting model system was normal (this is the
 equivalent of a non-vanishing discriminant for a single
 high order equation).  Other interesting systems of PDE,
 e.g. Maxwell's equations, do no have this property.

\bibliographystyle{plain}
\bibliography{transinvariant}

\end{document}